\documentclass[12pt, reqno]{amsart}
\usepackage{amsmath, amstext, amsbsy, amssymb, amscd}
\usepackage{amsmath}
\usepackage{amsxtra}
\usepackage{amscd}
\usepackage{amsthm}
\usepackage{amsfonts}
\usepackage{amssymb}
\usepackage{eucal}
\usepackage{color}
\usepackage{youngtab}

\input prepictex
\input pictex
\input postpictex

\newtheorem{theorem}{Theorem}[section]

\newtheorem{lemma}[theorem]{Lemma}
\newtheorem{proposition}[theorem]{Proposition}
\newtheorem{corollary}[theorem]{Corollary}
\theoremstyle{definition}
\newtheorem{definition}[theorem]{Definition}

\newtheorem{remark}[theorem]{Remark}
\definecolor{A}{rgb}{.75,1,.75}

\numberwithin{equation}{section}

\newcommand{\del}{\delta}
\newcommand{\Del}{\Delta}

\newcommand{\W}{\mathcal{W}}

\newcommand{\ev}{\operatorname{ev}}

\newcommand{\zomn}{$0^M1^N$}
\newcommand{\so}{\mathfrak{s}}
\newcommand{\mfs}{\check{\mathfrak{s}}}
\newcommand{\sd}{\mathfrak{s}^\dagger}
\newcommand{\sr}{\mathfrak{s}^r}
\newcommand{\gl}{\mathfrak{gl}}

\newcommand{\gr}{\operatorname{gr}}

\newcommand{\pa}[1]{|{#1}|}
\newcommand{\ovl}[1]{\overline{#1}}

\newcommand{\Z}{ \mathbb Z }

\begin{document}
\title[Parabolic presentations of $Y(\gl_{M|N})$] {Parabolic presentations of the super Yangian $Y(\gl_{M|N})$ associated with arbitrary 01-sequences}

\author[Yung-Ning Peng]{Yung-Ning Peng}
\address{Department of Mathematics, National Central University,
Chung-Li, Taiwan, 32054} \email{ynp@math.ncu.edu.tw}

\begin{abstract}
Let $\mu$ be an arbitrary composition of $M+N$ and let $\so$ be an arbitrary \zomn-sequence. A new presentation, depending on $\mu$ and $\so$, of the super Yangian $Y_{M|N}$ associated to the general linear Lie superalgebra $\gl_{M|N}$ is obtained.
\end{abstract}

\maketitle

\setcounter{tocdepth}{1}
\tableofcontents

\section{Introduction}

The Yangians, defined by Drinfeld \cite{Dr1, Dr2}, are certain non-commutative Hopf algebras that are important examples of quantum groups. They were studied to generate rational solutions of the {\em Yang-Baxter equation} and there were many applications in statistical mechanics and mathematical physics. Nowadays, the study of Yangians gives many new points of view and important applications to classical Lie theory; see the book \cite{Mo} and references therein.

Consider $Y_N=Y(\gl_N)$, the Yangian associated to the general Lie algebra $\gl_N$.
Associated to each composition $\mu$ of $N$, Brundan and Kleshchev established a parabolic presentation for $Y_N$ in \cite{BK1}. Roughly speaking, this new presentation of $Y_N$ corresponds to the Levi decomposition of $\gl_N$ with respect to $\mu$. In the special case when $\mu=(1,\ldots,1)$, the corresponding presentation is equivalent to Drinfeld's presentation (\cite[Remark 5.12]{BK1}). On the other extreme case when $\mu=(N)$, the corresponding presentation is called the {\em RTT presentation}; see \cite{MNO, Mo}.

The parabolic presentations play a fundamental role in their subsequent works. In \cite{BK2}, they established a concrete realization of finite $W$-algebras associated to {\em any} nilpotent element of type A in terms of Yangians, and a key step is to define a subalgebra of $Y_N$, called the {\em shifted Yangian}. Such a subalgebra can only be defined in terms of the new presentation found in \cite{BK1} except some special cases. The connection between Yangians and finite $W$-algebras was observed earlier in \cite{RS} for some particular nilpotent elements (called {\em rectangular} elements) with a different approach. Moreover, by means of such a realization, one may study the representation theory of finite $W$-algebras by studying the representation theory of Yangians; see \cite{BK3}.

The main goal of this article is to obtain the generalization of \cite{BK1} to the super Yangian $Y_{M|N}=Y(\gl_{M|N})$, the super Yangian associated to the general linear Lie superalgebra $\gl_{M|N}$. It is defined by Nazarov \cite{Na} in terms of the RTT presentation as a super analogue of $Y_{N}$.

One of the major differences between $\gl_N$ and $\gl_{M|N}$ is that,  in the case of $\gl_N$, we may always choose a canonical Borel subalgebra, since all of the Borel subalgebras are conjugated by the action of the Lie group $GL_N$. This Borel subalgebra gives a root system, and for this given root system we have multiple choices of simple systems. It suffices to choose a canonical simple system since they are all conjugate by the Weyl group action. However, the above argument is no longer true in the case of $\gl_{M|N}$. Therefore, in the study of $\gl_{M|N}$ and its representation theory, we may want to specifically mention which simple system we are using, and the notion of {\em 01-sequence} (called $\epsilon\delta$-sequences in \cite[Section 1.3]{CW}) is introduced as a parameterizing set of the conjugacy classes of simple systems of $\gl_{M|N}$ (also for some other types of Lie superalgebras) under the Weyl group action. 

For example, if we identify $\gl_{M|N}$ with the set of $(M+N)\times (M+N)$ matrices and choose the Cartan subalgebra $\mathfrak{h}$ to be the set of diagonal matrices, then the most common choice for $\mathfrak{s}$ is 
$$\so=\so^{st}=\stackrel{M}{\overbrace{0\ldots0}}\,\stackrel{N}{\overbrace{1\ldots1}},$$
where a representative simple system of $\gl_{M|N}$ in the corresponding Weyl group orbit is given by
\[
\Pi^{st}=\{ \delta_{i}-\delta_{i+1}, \epsilon_{j}-\epsilon_{j+1}, \delta_{M}-\epsilon_{1}\,|\, 1\leq i\leq M-1, 1\leq j\leq N-1\}.
\]
Here the notation $\delta_i$ and $\epsilon_j$ denote elements in $\mathfrak{h}^*$ such that
$\delta_{i}(X)$ equals to the $i$-th diagonal entry of $X\in\mathfrak{h}$, and $\epsilon_{j}(X)$ equals to the $(M+j)$-th diagonal entry of $X\in\mathfrak{h}$. With this standard choice, there is only one odd simple root and the behavior of $\gl_{M|N}$ is ``closest" to the classical $\gl_N$.

Note that the Weyl group is isomorphic to $S_M\times S_N$, which permutes those $\delta_i$'s and those $\epsilon_j$'s, respectively. Therefore, there is exactly one odd simple root, which is of the form $\delta_i-\epsilon_j$ for some $1\leq i\leq M$, $1\leq j\leq N$, in any other simple system in the Weyl group orbit of $\Pi^{st}$. It implies that the following two simple systems of $\gl_{3|2}$ are in different Weyl group orbits, since $\Pi_1$ contains only one simple odd root and $\Pi_2$ contains 4 odd simple roots:
\begin{align*}
\Pi_1&=\{\delta_1-\delta_2, \delta_2-\delta_3,\delta_3-\epsilon_1,\epsilon_1-\epsilon_2\} \longleftrightarrow  \mathfrak{s}_1=00011,\\
\Pi_2&=\{\delta_1-\epsilon_1, \epsilon_1-\delta_2,\delta_2-\epsilon_2,\epsilon_2-\delta_3\} \longleftrightarrow  \mathfrak{s}_2=01010.
\end{align*}
We refer the reader to \cite[Chapter 1.]{CW} for more details and further applications of 01-sequences.

It is noticed in \cite{Pe2} that the notion of $01$-sequence can be perfectly equipped to the RTT presentation of $Y_{M|N}$. It turns out that, up to an isomorphism, the definition of $Y_{M|N}$ is independent of the choices of the 01-sequence $\so$, and Nazarov's definition corresponds to the case when $\so=\so^{st}$ is the canonical one. Since the RTT presentation can be thought as merely a special case of the parabolic presentation by taking the composition $\mu=(M+N)$, the above observation suggests that it should be possible to obtain a corresponding parabolic presentation for {\em any} $\mu$, which triggered this work.

To be precise, the main result of this article (Theorem \ref{Pg}) is that for an {\em arbitrary} fixed 01-sequence $\so$ of $\gl_{M|N}$ and an {\em arbitrary} fixed composition $\mu$ of $M+N$, a presentation of $Y_{M|N}$ is obtained.

We quickly explain the idea, which is basically generalizing the argument in \cite{BK1} and adapting some techniques in \cite{Go, Pe1} dealing with the sign factors.
Fix a composition $\mu$ of $M+N$ and fix an arbitrary $01$-sequence $\so$ of $\gl_{M|N}$. We first define some distinguished elements in $Y_{M|N}$ associated to $\mu$, denoted by $D$'s, $E$'s and $F$'s, by {\em Gauss decomposition} (or {\em quasideterminants}). 

Roughly speaking, the elements $D$'s are those elements in the diagonal blocks of the block matrix decomposition of $Y_{M|N}$ with respect to $\mu$, while the $E$'s and the $F$'s are those elements in the upper and lower diagonal blocks, respectively. Note that these elements depend on the shape $\mu$, where their parities are determined by the given $\so$. These elements form a generating set for $Y_{M|N}$ (Theorem \ref{gendef}), so the next step is to find enough relations to achieve a presentation.

In the case of \cite{BK1}, if the generators are from two different blocks and the blocks are not ``close", then they commute. Fortunately, this phenomenon remains to be true under our general setting (Lemma~\ref{corcommute}) and it enormously reduces the number of the non-vanishing relations. As a consequence, we only have to focus on the supercommutation relations of the elements that are either in the same block, or their belonging blocks are ``close enough". Let $n$ be the length of $\mu$. When $n=2,3,4$, the situations are less complicated so that we may derive various relations among those generators by direct computation.

Next, we take advantage of the homomorphisms $\psi_L$ and $\zeta_{M|N}$ between super Yangians (see Section 4). These maps carry the relations in the special cases (with $n\leq 4$) to the general case, so that we obtain many relations in $Y_{M|N}$. Finally we prove that we have found enough relations for our presentation.

As a matter of fact, there are already a few results \cite{Go, Pe1} on such a generalization focusing on the canonical case when $\so=\so^{st}$. 
In \cite{Go}, a presentation of $Y_{M|N}$ when $\mu=(1^{M+N})$ is obtained, which is a generalization of Drinfeld's presentation. In \cite{Pe1}, a similar result when $\mu$ is of the form $\mu=(\lambda,\nu)$, where $\lambda$ is a composition of $M$ and $\nu$ is a composition of $N$, is obtained.
However, the results only explained the case when $\so=\so^{st}$. Moreover, the compositions $\mu$ therein are very special so that the elements in one block must have the same parity so the super phenomenon only happens at a few specific places. 

Under our setting, $\mu$ and $\so$ are {\em both arbitrary} so that an even element and an odd element could exist in the same block. As a result, the super phenomenon could happen everywhere. Hence in our current consideration, the signs arising from the $\mathbb{Z}_2$-grading are much more involved than \cite{Go, Pe1} and one needs more elaborate notation and extra care when treating the sign issues. Roughly speaking, we need to correctly insert the necessary sign factors in almost every formula. Certainly, our main theorem covers the above results as special cases.

Finally we mention one undergoing application of our result, which can also be thought as the true motivation of this work. Following the classical $\gl_N$ case, one may try to generalize the argument in \cite{BK2} so that a realization of finite $W$-superalgebras of type A in terms of the super Yangian $Y_{M|N}$ can be obtained. Such a connection was observed in \cite{BR} for {\em rectangular} nilpotent elements, but this is still open for the general nilpotent case.

In fact, based on \cite{Go, Pe1}, there are already some partial results \cite{BBG, Pe2, Pe3} about the the realization of finite $W$-superalgebras when the nilpotent element is {\em principal} or satisfying certain restrictions. As noticed in \cite{BBG, Pe3}, if we want to generalize the argument in \cite{BK2} to the case of $Y_{M|N}$ in full generality, then a more general presentation of $Y_{M|N}$ is required. One of the reasons is that the {\em shifted super Yangian}, which is a subalgebra of $Y_{M|N}$, can be defined only under some nice assumptions as in \cite{BBG, Pe3}. With our new presentations, the shifted super Yangian can be defined for {\em any} given nilpotent element in $\gl_{M|N}$ so that it is possible to establish the connection in full generality, and this is currently in progress by the author.

This article is organized in the following fashion. In Section 2, we recall some basic properties of $Y_{M|N}$. 
In Section 3, we explicitly define the parabolic generators by means of Gauss decomposition and show that they indeed form a generating set. 
In Section 4, we define some homomorphisms between super Yangians so that we may reduce the general case to some less complicated special cases. Section 5 and 6 are devoted to further study about these special cases. Our main theorem is formulated and proved in Section 7.

\section{Prelimilaries}
Let $\so$ be a \zomn-sequence (or 01-sequence for short) of $\gl_{M|N}$, which is a sequence consisting of $M$ $0$'s and $N$ $1$'s, arranged in a row with respect to a certain order. It is well-known \cite[Proposition 1.27]{CW} that there is a bijection between the set of $0^M1^N$-sequence and the Weyl group orbits of simple systems of $\gl_{M|N}$.

For homogeneous elements $A$ and $B$ in a $\Z_2$-graded algebra $L$, the {\em supercommutator of $A$ and $B$} is defined by 
\[
\big[ A,B \big] = AB-(-1)^{\pa{A}\pa{B}}BA,
\]
where $\pa{A}$ is the $\Z_2$-grading of $A$ in $L$, or called the {\em parity} of $A$. By convention, a homogeneous element $A$ is called $even$ if $\pa{A}=0$, and called $odd$ if $\pa{A}=1$. For each $1\leq i\leq M+N$, let $\pa{i}$ denote the $i$-th digit of the fixed \zomn-sequence $\so$.

\begin{definition}
For a given $\so$, the super Yangian associated to the general linear Lie superalgebra $\gl_{M|N}$, denoted by $Y_{M|N}$ hereafter, is the associative $\mathbb{Z}_2$-graded algebra (i.e., superalgebra) over $\mathbb{C}$ generated by the {\em RTT generators} \cite{Na}
\begin{equation}\label{RTTgen}
\left\lbrace t_{i,j}^{(r)}\,| \; 1\le i,j \le M+N; r\ge 1\right\rbrace,
\end{equation}
subject to following relations:
\begin{equation}\label{RTT}
\big[ t_{i,j}^{(r)}, t_{h,k}^{(s)} \big] = (-1)^{\pa{i}\,\pa{j} + \pa{i}\,\pa{h} + \pa{j}\,\pa{h}}
\sum_{g=0}^{\mathrm{min}(r,s) -1} \Big( t_{h,j}^{(g)}\, t_{i,k}^{(r+s-1-g)} -   t_{h,j}^{(r+s-1-g)}\, t_{i,k}^{(g)} \Big),
\end{equation}
where the parity of $t_{i,j}^{(r)}$ for $r>0$ is defined by $\pa{i}+\pa{j}$ (mod 2), and the bracket is understood as the supercommutator. By convention, we set $t_{i,j}^{(0)}:=\del_{ij}$.
\end{definition}
Similar to the $\gl_{M|N}$ case, for $r>0$, the element $t_{i,j}^{(r)}$ is called an $even$ ($odd$, respectively) element if its parity is $0$ ($1$, respectively). 
The original definition in \cite{Na} corresponds to the case when $\so$ is the canonical one; that is, $\so$ is of the form
\[
\so=\so^{st}:=\stackrel{M}{\overbrace{0\ldots0}}\,\stackrel{N}{\overbrace{1\ldots1}}.
\]
As observed in \cite{Pe2}, the definition of $Y_{M|N}$ is independent of the choices of $\so$, up to an isomorphism, so we often omit it in the notation.

For each $1\leq i,j\leq M+N$, define the formal power series 
\[
t_{i,j}(u):= \sum_{r\geq 0} t_{i,j}^{(r)}u^{-r} \in Y_{M|N}[[u^{-1}]].
\]
It is well-known \cite[p.125]{Na} that $Y_{M|N}$ is a Hopf-superalgebra, where the comultiplication 
$\Del:Y_{M|N}\rightarrow Y_{M|N}\otimes Y_{M|N}$ is given by 
\begin{equation}\label{Del}
\Del(t_{i,j}^{(r)})=\sum_{s=0}^r \sum_{k=1}^{M+N} t_{i,k}^{(r-s)}\otimes t_{k,j}^{(s)}.
\end{equation}
Moreover, there exists a surjective homomorphism $$\ev:Y_{M|N}\rightarrow U(\gl_{M|N})$$ called the {\em evaluation homomorphism}, defined by
\begin{equation}\label{ev}
\ev\big(t_{i,j}(u)\big):= \del_{ij} + (-1)^{|i|} e_{ij}u^{-1},
\end{equation}
where $e_{ij}\in\gl_{M|N}$ is the elementary matrix.

For homogeneous elements $x_1,\ldots,x_s$ in a superalgebra $A$, a {\em supermonomial} in $x_1,\ldots,x_s$ means a monomial of the form $x_1^{i_1}\cdots x_s^{i_s}$ for some $i_1,\ldots,i_s\in \mathbb{Z}_{\geq 0}$ and $i_j\leq 1$ if $x_j$ is odd. The following proposition is a PBW theorem for $Y_{M|N}$, where the proof in \cite{Go} works perfectly for any fixed $\so$.
\begin{proposition}\cite[Theorem 1]{Go}\label{PBWSY}
The set of supermonomials in the following elements of $Y_{M|N}$
\[
\left\lbrace t_{i,j}^{(r)}\, |\, 1\leq i,j\leq M+N,  r\geq 1 \right\rbrace
\] taken in some fixed order
forms a linear basis for $Y_{M|N}$.
\end{proposition}

Define the $loop$ $filtration$ on $Y_{M|N}$
\begin{equation}\label{filt2}\notag
L_0 Y_{M|N} \subseteq L_1 Y_{M|N} \subseteq L_2 Y_{M|N} \subseteq \cdots
\end{equation}
by setting $\deg t_{ij}^{(r)}=r-1$ for each $r\geq 1$ and let $L_kY_{M|N}$ be the span of all supermonomials of the form 
$$t_{i_1j_1}^{(r_1)}t_{i_2j_2}^{(r_2)}\cdots t_{i_sj_s}^{(r_s)}$$
with total degree not greater than $k$. The associated graded superalgebra is denoted by $\gr Y_{M|N}.$

Let $\gl_{M|N}[x]$ denote the {\em loop superalgebra} $\gl_{M|N}\otimes \mathbb{C}[x]$ with the standard basis 
$$\lbrace e_{ij}x^r \,|\, 1\leq i,j\leq M+N, r\geq 0\rbrace$$ 
and let $U(\gl_{M|N}[x])$ denote its universal enveloping algebra. The next corollary follows from Proposition~\ref{PBWSY}.

\begin{corollary}\cite[Corollary 1]{Go}\label{Yloop}
The graded superalgebra $\gr Y_{M|N}$ is isomorphic to $U(\gl_{M|N}[x])$ by the map
\begin{center}
$\gr Y_{M|N}\rightarrow U(\gl_{M|N}[x])$\\
$\gr_{r-1}t_{ij}^{(r)}\mapsto (-1)^{\pa{i}}e_{ij}x^{r-1}.$
\end{center}
\end{corollary}

\section{Parabolic generators}
 Let $\mu=(\mu_1,\ldots,\mu_n)$ be a given composition of $M+N$ with length $n$ and fix a \zomn-sequence $\so$. We break $\so$ into $n$ subsequences according to $\mu$; that is, 
\[\so=\so_1\so_2\ldots\so_n,\] 
where $\so_1$ is the subsequence consisting of the first $\mu_1$ digits of $\so$, $\so_2$ is the subsequence consisting of the next $\mu_2$ digits of $\so$, and so on.
For example, if we have $\so=011100011$ and $\mu=(2,4,3)$, then
\[
\so=\overbrace{01}^{\so_1} \, \overbrace{1100}^{\so_2} \, \overbrace{011}^{\so_3}.
\]

For each $1\leq a\leq n$, let $p_a$ and $q_a$ denote the number of $0$'s and $1$'s in $\so_a$, respectively. By definition, each $\so_a$ is a $0^{p_a}1^{q_a}$-sequence of $\gl_{p_a|q_a}$. Moreover, for all $1\leq i\leq \mu_a$, define the {\em restricted parity} $\pa{i}_a$ by
\begin{center} $\pa{i}_a$:= the $i$-th digits of $\so_a$. \end{center}
By definition, for each $1\leq a\leq n$ and $1\leq i\leq \mu_a$, we have
\begin{equation}\label{respa}
\pa{i}_a=\pa{\sum_{j=1}^{a-1}\mu_j+i}.
\end{equation}
Basically, the techniques in \cite{BK1, Pe1} work perfectly with the notion of an arbitrary $\so$ that we just introduced earlier. In order to make this article self-contained, we spend some time explaining the notation precisely.

Define the $(M+N)\times (M+N)$ matrix with entries in $Y_{M|N}[[u^{-1}]]$ by
\[
T(u):=\Big( t_{i,j}(u) \Big)_{1\leq i,j\leq M+N}
\]
Note that for any fixed $\so$, the leading minors of the matrix $T(u)$ are invertible and hence it possesses a $Gauss$ $decomposition$ \cite{GR} with respect to $\mu$; that is,
\begin{equation}\label{T=FDE}
T(u) = F(u) D(u) E(u)
\end{equation}
for unique {\em block matrices} $D(u)$, $E(u)$ and $F(u)$ of the form
$$
D(u) = \left(
\begin{array}{cccc}
D_{1}(u) & 0&\cdots&0\\
0 & D_{2}(u) &\cdots&0\\
\vdots&\vdots&\ddots&\vdots\\
0&0 &\cdots&D_{n}(u)
\end{array}
\right),
$$

$$
E(u) =
\left(
\begin{array}{cccc}
I_{\mu_1} & E_{1,2}(u) &\cdots&E_{1,n}(u)\\
0 & I_{\mu_2} &\cdots&E_{2,n}(u)\\
\vdots&\vdots&\ddots&\vdots\\
0&0 &\cdots&I_{\mu_{n}}
\end{array}
\right),\:
$$

$$
F(u) = \left(
\begin{array}{cccc}
I_{\mu_1} & 0 &\cdots&0\\
F_{2,1}(u) & I_{\mu_2} &\cdots&0\\
\vdots&\vdots&\ddots&\vdots\\
F_{n,1}(u)&F_{n,2}(u) &\cdots&I_{\mu_{n}}
\end{array}
\right),
$$
where
\begin{align}
D_a(u) &=\big(D_{a;i,j}(u)\big)_{1 \leq i,j \leq \mu_a},\\
E_{a,b}(u)&=\big(E_{a,b;i,j}(u)\big)_{1 \leq i \leq \mu_a, 1 \leq j \leq \mu_b},\\
F_{b,a}(u)&=\big(F_{b,a;i,j}(u)\big)_{1 \leq i \leq \mu_b, 1 \leq j \leq \mu_a},
\end{align}
are $\mu_a \times \mu_a$,
$\mu_a \times \mu_b$
and  $\mu_b \times\mu_a$ matrices, respectively, for all $1\le a\le n$ in (3.3)
and all $1\le a<b\le n$ in (3.4) and (3.5).
Also note that all the submatrices $D_{a}(u)$'s are invertible, and we define the $\mu_a\times\mu_a$ matrix
$D_a^{\prime}(u)=\big(D_{a;i,j}^{\prime}(u)\big)_{1\leq i,j\leq \mu_a}$ by
\begin{equation*}
D_a^{\prime}(u):=\big(D_a(u)\big)^{-1}.
\end{equation*}
The entries of these matrices are certain power series
\begin{eqnarray*}
&D_{a;i,j}(u) =& \sum_{r \geq 0} D_{a;i,j}^{(r)} u^{-r},\\
&D_{a;i,j}^{\prime}(u) =& \sum_{r \geq 0}D^{\prime(r)}_{a;i,j} u^{-r},\\
&E_{a,b;i,j}(u) =& \sum_{r \geq 1} E_{a,b;i,j}^{(r)} u^{-r},\\
&F_{b,a;i,j}(u) =& \sum_{r \geq 1} F_{b,a;i,j}^{(r)} u^{-r}.
\end{eqnarray*}
In addition, for $1\leq a\leq n-1$, we set
\begin{eqnarray*}
&E_{a;i,j}(u) :=& E_{a,a+1;i,j}(u)=\sum_{r \geq 1}  E_{a;i,j}^{(r)} u^{-r},\\
&F_{a;i,j}(u) :=& F_{a+1,a;i,j}(u)=\sum_{r \geq 1} F_{a;i,j}^{(r)} u^{-r}.
\end{eqnarray*}

We will show that the coefficients of these series forms a generating set for $Y_{M|N}$.
As a matter of fact, one may describe all these series in terms of the RTT generators explicitly by \textit{quasideterminants} \cite{GR}. Here we follow the notation in \cite[(4.3)]{BK1}. Suppose that $A, B, C$ and $D$ are $a \times a$, $a \times b$, $b \times a$ and $b \times b$ matrices respectively with entries in some ring.
Assuming that the matrix $A$ is invertible, we define
\begin{equation*}
\left|
\begin{array}{cc}
A&B\\
C&
\hbox{\begin{tabular}{|c|}\hline$D$\\\hline\end{tabular}}
\end{array}
\right| := D - C A^{-1} B.
\end{equation*}
Write the matrix $T(u)$ in block form according to $\mu$ as
$$
T(u) = \left(
\begin{array}{lll}
{^\mu}T_{1,1}(u)&\cdots&{^\mu}T_{1,n}(u)\\
\vdots&\ddots&\cdots\\
{^\mu}T_{n,1}(u)&\cdots&{^\mu}T_{n,n}(u)\\
\end{array}
\right),
$$ where each ${^\mu}T_{a,b}(u)$ is a $\mu_a \times \mu_b$ matrix.

\begin{proposition}\cite{GR}\label{quasi} We have
\begin{align}\label{quasid}
&D_a(u) =
\left|
\begin{array}{cccc}
{^\mu}T_{1,1}(u) & \cdots & {^\mu}T_{1,a-1}(u)&{^\mu}T_{1,a}(u)\\
\vdots & \ddots &\vdots&\vdots\\
{^\mu}T_{a-1,1}(u)&\cdots&{^\mu}T_{a-1,a-1}(u)&{^\mu}T_{a-1,a}(u)\\
{^\mu}T_{a,1}(u) & \cdots & {^\mu}T_{a,a-1}(u)&
\hbox{\begin{tabular}{|c|}\hline${^\mu}T_{a,a}(u)$\\\hline\end{tabular}}
\end{array}
\right|,\\[4mm]
&E_{a,b}(u) =\label{quasie}
D^{\prime}_a(u)
\left|\begin{array}{cccc}
{^\mu}T_{1,1}(u) & \cdots &{^\mu}T_{1,a-1}(u)& {^\mu}T_{1,b}(u)\\
\vdots & \ddots &\vdots&\vdots\\
{^\mu}T_{a-1,1}(u) & \cdots & {^\mu}T_{a-1,a-1}(u)&{^\mu}T_{a-1,b}(u)\\
{^\mu}T_{a,1}(u) & \cdots & {^\mu}T_{a,a-1}(u)&
\hbox{\begin{tabular}{|c|}\hline${^\mu}T_{a,b}(u)$\\\hline\end{tabular}}
\end{array}
\right|,\\[4mm]
&F_{b,a}(u) =\label{quasif}
\left|
\begin{array}{cccc}
{^\mu}T_{1,1}(u) & \cdots &{^\mu}T_{1,a-1}(u)& {^\mu}T_{1,a}(u)\\
\vdots & \ddots &\vdots&\vdots\\
{^\mu}T_{a-1,1}(u) & \cdots & {^\mu}T_{a-1,a-1}(u)&{^\mu}T_{a-1,a}(u)\\
{^\mu}T_{b,1}(u) & \cdots & {^\mu}T_{b,a-1}(u)&
\hbox{\begin{tabular}{|c|}\hline${^\mu}T_{b,a}(u)$\\\hline\end{tabular}}
\end{array}
\right|D^{\prime}_a(u),
\end{align}
for all $1\leq a\leq n$ in (\ref{quasid}) and $1\leq a<b\leq n$ in (\ref{quasie}), (\ref{quasif}).
\end{proposition}

Let $T_{a,b;i,j}(u)$ be the $(i,j)$-th entry of the $\mu_a\times\mu_b$ matrix $^\mu T_{a,b}(u)$ and let $T_{a,b;i,j}^{(r)}$ denote the coefficient of $u^{-r}$ in $T_{a,b;i,j}(u)$. As a consequence of Proposition \ref{quasi}, we have
\begin{equation}\label{Jacobi1}
E_{b-1;i,j}^{(1)}=T_{b-1,b;i,j}^{(1)}, \qquad F_{b-1;j,i}^{(1)}=T_{b,b-1;j,i}^{(1)},
\end{equation}
for all $2\leq b\leq n$, all $1\leq i\leq \mu_{b-1}, 1\leq j\leq \mu_{b}$.
In particular,
\begin{equation}\label{Dtident}
D_{1;i,j}^{(r)}=T_{1,1;i,j}^{(r)}=t_{i,j}^{(r)},\quad \text{for all}\quad 1\leq i,j\leq \mu_1, r\geq 0.
\end{equation}

Note that the matrix $T(u)$ is always invertible and we define the entries of its inverse by
\[
\big(T(u)\big)^{-1}:=\big(t_{ij}^{\prime}(u)\big)_{i,j=1}^{M+N}.
\]
Taking inverse to the matrix equation (\ref{T=FDE}), we have
\begin{equation}\label{Tp=FDE}
t_{i,j}^{\prime}(u) = \big( F(u)^{-1} D(u)^{-1} E(u)^{-1} \big)_{i,j}
\end{equation}

Applying the technique in \cite[Section 1]{MNO}, we may rewrite equation (\ref{RTT}) into the following series form
$$
(u-v)[t_{ij}(u),t_{hk}(v)]
=(-1)^{\pa{i}\,\pa{j}+\pa{i}\,\pa{h}+\pa{j}\,\pa{h}}
\Big(t_{hj}(u)t_{ik}(v)-t_{hj}(v)t_{ik}(u)\Big).
$$
The next lemma, which will be used frequently later, can be deduced from the above equation by a similar calculation as in \cite[Section 2]{Pe1}.
\begin{lemma}
\begin{multline}\label{usefull}
(u-v)[t_{ij}(u),t'_{hk}(v)]
=(-1)^{\pa{i}\,\pa{j}+\pa{i}\,\pa{h}+\pa{j}\,\pa{h}}\times\\
\Big(\delta_{h,j}
\sum_{g=1}^{M+N}
t_{ig}(u)t'_{gk}(v)-\delta_{i,k}\sum_{g=1}^{M+N}t'_{hg}(v)t_{gj}(u)\Big),
\end{multline}
for all $1\leq i,j,h,k\leq M+N$.
\end{lemma}

\begin{lemma}
For each pair $a$, $b$ such that $1<a+1<b\leq n-1$ and $1\leq i\leq\mu_a$, $1 \leq j \leq \mu_b$, we have
\begin{equation}
E_{a,b;i,j}^{(r)} = (-1)^{\pa{k}_{b-1}}[E_{a,b-1;i,k}^{(r)}, E_{b-1;k,j}^{(1)}],
\,\,\,\,\,
F_{b,a;j,i}^{(r)} = (-1)^{\pa{k}_{b-1}}[F_{b-1;j,k}^{(1)}, F_{b-1,a;k,i}^{(r)}],\label{ter}
\end{equation} 
for any fixed $1 \leq k \leq \mu_{b-1}$. 
\end{lemma}
\begin{proof}
This can be proved by induction on $b-a>1$. We perform the initial step when $a=1$ and $b=3$ for the $E$'s here, while the general case and the statement for the $F$'s can be established in a similar way.

By (\ref{quasie}) and (\ref{Jacobi1}), we have 
$$[E_{1,2;i,h}^{(r)}, E_{2,3;k,j}^{(1)}]=[\sum_{p=1}^{\mu_1}\sum_{s=0}^{r}D_{1;i,p}^{\prime(s)} T_{1,2;p,h}^{(r-s)} ,  T_{2,3;k,j}^{(1)}].$$
Note that we may express $D_{1;i,p}^{\prime(r)}$ in terms of $t_{\alpha,\beta}^{\prime(r)}$, and the subscriptions $1\leq \alpha,\beta\leq \mu_1$ will never overlap with the subscriptions of $T_{2,3;k,j}^{(1)}=t_{\mu_1+k,\mu_1+\mu_2+j}^{(1)}$, where $1\leq k\leq \mu_2, 1\leq j\leq \mu_3$. Thus they supercommute by (\ref{usefull}), and the right-hand side of the equation equals to
$$\sum_{p=1}^{\mu_1}\sum_{s=0}^{r}D_{1;i,p}^{\prime(s)} [ T_{1,2;p,h}^{(r-s)}, T_{2,3;k,j}^{(1)}].$$ 
The bracket can be computed by (\ref{RTT}):
$$[ T_{1,2;p,h}^{(r-s)},  T_{2,3;k,j}^{(1)}]=[t_{p,\mu_1+h}^{(r-s)}, t_{\mu_1+k,\mu_1+\mu_2+j}^{(1)}]=(-1)^{\pa{k}_2} \delta_{h,k} t_{p,\mu_1+\mu_2+j}^{(r-s)}.$$
Thus, for any $1\leq k\leq \mu_2$, we have
$$(-1)^{\pa{k}_2}[E_{1,2;i,k}^{(r)}, E_{2,3;k,j}^{(1)}]=\sum_{p=1}^{\mu_1}\sum_{s=0}^{r}D_{1;i,p}^{\prime(s)}t_{p,\mu_1+\mu_2+j}^{(r-s)}=\sum_{p=1}^{\mu_1}\sum_{s=0}^{r}D_{1;i,p}^{\prime(s)}T_{1,3;p,j}^{(r-s)}.$$
By (\ref{quasie}), the right-hand side of the above equation is exactly $E_{1,3;i,j}^{(r)}$. The general case is similar, except that the expression of $E_{a,b-1;i,h}^{(r)}$ by (\ref{quasie}) is more complicated.
\end{proof}

\begin{theorem}\label{gendef}
The superalgebra $Y_{M|N}$ is generated by the following elements
\begin{align*}
&\big\lbrace D_{a;i,j}^{(r)}, D_{a;i,j}^{\prime(r)} \,|\, {1\leq a\leq n,\; 1\leq i,j\leq \mu_a,\; r\geq 0}\big\rbrace,\\
&\big\lbrace E_{a;i,j}^{(r)} \,|\, {1\leq a< n,\; 1\leq i\leq \mu_a, 1\leq j\leq\mu_{a+1},\; r\geq 1}\big\rbrace,\\
&\big\lbrace F_{a;i,j}^{(r)} \,|\, {1\leq a< n,\; 1\leq i\leq\mu_{a+1}, 1\leq j\leq \mu_a,\; r\geq 1}\big\rbrace.
\end{align*}
\end{theorem}
\begin{proof}
Multiplying the matrix equation (\ref{T=FDE}), we see that each $t_{ij}^{(r)}$ can be expressed as a sum of supermonomials in $D_{a;i,j}^{(r)}$, $E_{a,b;i,j}^{(r)}$ and $F_{b,a;i,j}^{(r)}$, in a certain order that all the $F$'s appear before the $D$'s and all the $D$'s appear before the $E$'s. By (\ref{ter}), it is enough to use $D_{a;i,j}^{(r)}$, $E_{a;i,j}^{(r)}$ and $ F_{a;i,j}^{(r)}$ only rather than use all of the $E$'s and the $F$'s.
\end{proof}

\begin{remark}
Actually we don't need those $D_{a;i,j}^{\prime(r)}$ to obtain a generating set. 
They only appear in (\ref{p702}) and (\ref{p706}). One can in fact remove (\ref{p702}) from the defining relations and rewrite (\ref{p706}) into a new expression which is free from those $D^{\prime(r)}_{a;i,j}$.
The reason to include them is to make the presentation look concise.
\end{remark}

The generators of $Y_{M|N}$ in Theorem \ref{gendef} above are called {\em parabolic generators}. Note that these generators depend on the shape $\mu$ and their parities depend on the fixed sequence $\so$ (see (\ref{pad})-(\ref{paf}) later). We will use the notation $Y_\mu$ or $Y_{M|N}(\so)$ or $Y_\mu(\so)$ to emphasize the choice of $\mu$ or $\so$ or both when necessary. The goal of this article is to write down explicitly a set of defining relations of $Y_\mu(\so)$ with respect to the parabolic generators for {\em any} fixed $\mu$ and {\em any} fixed $\so$.

\section{Homomorphisms between super Yangians}
To explicitly write down the relations among the parabolic generators in Theorem \ref{gendef}, we start with the special cases when $n$ are either 2 or 3, that are relatively less complicated. The other relations in full generality can be deduced from these special ones by applying certain nice injective homomorphisms that we are about to introduce.

We start with some notation. Let $\so$ be a fixed $0^M1^N$-sequence. We define 
\begin{itemize}
\item $\mfs$:= the $0^N1^M$-sequence obtained by interchanging the 0's and 1's of $\so$.
\item $\sr$:= the reverse of $\so$.
\item $\sd$:= $(\mfs)^r$, the reverse of $\mfs$.
\end{itemize}
For instance, if $\so=0011010$, then $\mfs=1100101$, $\sr=0101100$, and $\sd=1010011$.
Moreover, if $\so_1$ and $\so_2$ are two 01-sequences, then $\so_1\so_2$ simply means the concatenation of $\so_1$ and $\so_2$.

The following proposition is a generalization of \cite[Proposition 4.1]{Pe1} to an arbitrary 01-sequence $\so$, where the proof is similar so we omit; see also \cite[Section~4]{Go}
\begin{proposition}
\begin{enumerate}
  \item [1.]The map $\rho_{M|N}:Y_{M|N}(\so)\rightarrow Y_{N|M}(\sd)$ defined by
  \[
  \rho_{M|N}\big(t_{ij}(u)\big)=t_{M+N+1-i,M+N+1-j}(-u)
  \]
   is an isomorphism.
  \item [2.]The map $\omega_{M|N}:Y_{M|N}(\so)\rightarrow Y_{M|N}(\so)$ defined by
  \[
  \omega_{M|N}\big(T(u)\big)=\big(T(-u)\big)^{-1}
  \]
   is an automorphism.
  \item [3.]The map $\zeta_{M|N}:Y_{M|N}(\so)\rightarrow Y_{N|M}(\sd)$ defined by
  \[
  \zeta_{M|N}=\rho_{M|N}\circ\omega_{M|N}
  \]
  is an isomorphism.
  \item [4.] Let $p,q\in\mathbb{Z}_{\ge0}$ be given and let $\so_1$ be an arbitrary $0^p1^q$-sequence. 
  Let 
  $$\varphi_{p|q}:Y_{M|N}(\so)\rightarrow Y_{p+M|q+N}(\so_1\so)$$ 
  be the injective homomorphism sending each $t_{i,j}^{(r)}$ in $Y_{M|N}(\so)$ to $t_{p+q+i,p+q+j}^{(r)}$ in $Y_{p+M|q+N}(\so_1\so)$. Then
  the map $\psi_{p|q}:Y_{M|N}(\so)\rightarrow Y_{p+M|q+N}(\so_1\so)$ defined by
  \[
  \psi_{p|q}=\omega_{p+M|q+N}\circ\varphi_{p|q}\circ\omega_{M|N},
  \]
   is an injective homomorphism.
 \end{enumerate}
\end{proposition}

In fact, only the maps $\zeta_{M|N}$ and $\psi_{p|q}$  will be used later so we write down their images explicitly.
\begin{lemma}
Let $1\leq i,j \leq M+N$.
\begin{enumerate}
  \item [1.]For any $p,q\in\mathbb{Z}_{\geq 0}$, we have \begin{equation}\label{psit}
  \psi_{p|q}\big(t_{ij}(u)\big)=
  \left| \begin{array}{cccc} t_{1,1}(u) &\cdots &t_{1,p+q}(u) &t_{1, p+q+j}(u)\\
         \vdots &\ddots &\vdots &\vdots \\
         t_{p+q,1}(u) &\cdots &t_{p+q,p+q}(u) &t_{p+q, p+q+j}(u)\\
         t_{p+q+i, 1}(u) &\cdots &t_{p+q+i,p+q}(u) &\boxed{t_{p+q+i, p+q+j}(u)}
         \end{array} \right|.
         \end{equation}
  \item [2.] \begin{equation}\label{zetat}
        \zeta_{M|N}\big(t_{ij}(u)\big)=t_{M+N+1-i,M+N+1-j}^{\prime}(u).
        \end{equation}
\end{enumerate}
\end{lemma}
\begin{proof}
The first one is exactly the same with \cite[Lemma 4.2]{BK1}, while the second one is immediate from the definition.\end{proof}

Set $L=p+q$ for convenience. Note that (\ref{psit}) depends only on $L$ so we may simply write $\psi_{p|q}=\psi_L$ when appropriate. As a consequence of Proposition \ref{quasi}, we have
 \begin{eqnarray}\label{psid}
  D_{a;i,j}(u)&=&\psi_{\mu_1+\mu_2+\ldots +\mu_{a-1}}\big(D_{1;i,j}(u)\big),\\ \label{psie} 
  E_{a;i,j}(u)&=&\psi_{\mu_1+\mu_2+\ldots +\mu_{a-1}}\big(E_{1;i,j}(u)\big),\\ \label{psif}
  F_{a;i,j}(u)&=&\psi_{\mu_1+\mu_2+\ldots +\mu_{a-1}}\big(F_{1;i,j}(u)\big).
 \end{eqnarray}
In addition, the map $\psi_{L}$ sends $t_{i,j}^{\prime}(u)\in Y_{M|N}(\so)$ to $t_{L+i,L+j}^{\prime}(u)\in Y_{p+M|q+N}(\so_1\so)$, which implies that the image of $\psi_{L}\big(Y_{M|N}(\so)\big)$ in $Y_{p+M|q+N}(\so_1\so)$ is generated by the following elements 
\[
\{t_{L+i,L+j}^{\prime(r)}\in Y_{p+M|q+N}(\so_1\so)\,|\,1\leq i,j\leq M+N, r\geq 0\}.
\]

If we pick any element $t_{i,j}^{(r)}$ in the northwestern $L\times L$ corner of $T(u)$, an $(L+M+N)\times(L+M+N)$ matrix with entries in $Y_{p+M|q+N}[[u^{-1}]]$, then the indices will never overlap with those of $\psi_{L}\big(Y_{M|N}\big)$, that are in the southeastern $(M+N)\times(M+N)$ corner of the same $T(u)$. As a result of equation (\ref{usefull}), they supercommute. Clearly, the elements in the northwestern $L\times L$ corner of $T(u)$ (of $Y_{p+M|q+N}$) generate a subalgebra which is isomorphic to $Y_{p|q}(\so_1)$ by (\ref{RTT}), so we have obtained the following lemma. Roughly speaking, the map $\psi_{L}$ embeds the matrix $T(u)$ into the southeastern corner of a larger matrix $T(u)$.
\begin{lemma}\label{corcommute}
In $Y_{p+M|q+N}(\so_1\so)$, the subalgebras $Y_{p|q}(\so_1)$ and $\psi_{L}\big(Y_{M|N}(\so)\big)$ supercommute with each other.
\end{lemma}

Moreover, by equations (\ref{psit}), (\ref{psid}), (\ref{psie}) and (\ref{psif}), the parities of the parabolic generators can be easily obtained as follows:
\begin{eqnarray}
    \label{pad}\text{ parity of } D_{a;i,j}^{(r)}&=&\pa{i}_a+\pa{j}_a \,\,\text{(mod 2)},\\
    \label{pae}\text{ parity of } E_{b;h,k}^{(s)}&=&\pa{h}_{b}+\pa{k}_{b+1} \,\,\text{(mod 2)},\\
    \label{paf}\text{ parity of } F_{b;f,g}^{(s)}&=&\pa{f}_{b+1}+\pa{g}_{b} \,\,\text{(mod 2)},
\end{eqnarray}
for all $r,s\geq 1$, $1\leq a\leq n$, $1\leq b\leq n-1$, $1\leq i,j\leq \mu_a$, $1\leq h,g\leq \mu_b$, $1\leq k,f\leq \mu_{b+1}$.

Next we analyze $\zeta_{M|N}$. Associate to $\mu$, we may define the elements \{$D_{a;i,j}^{(r)};D^{\prime(r)}_{a;i,j}$\}, \{$E_{a;i,j}^{(r)}$\}, \{$F_{a;i,j}^{(r)}$\} in $Y_{M|N}=Y_\mu(\so)$ by Gauss decomposition. Consider
\[\overleftarrow{\mu}:=(\mu_n,\mu_{n-1},\ldots,\mu_{1}),\]
the reverse of $\mu$.
We may similarly define elements \{$\overleftarrow{D}_{a;i,j}^{(r)};\overleftarrow{D}^{\prime(r)}_{a;i,j}$\}, \{$\overleftarrow{E}_{a;i,j}^{(r)}$\}, \{$\overleftarrow{F}_{a;i,j}^{(r)}$\} in $Y_{N|M}=Y_{\overleftarrow{\mu}}(\sd)$ with respect to $\overleftarrow{\mu}$ in the same way. Their relations are described in the following proposition, which is a generalization of \cite[Proposition~1]{Go} and \cite[Proposition~4.4]{Pe1}, where the proof is very similar. Roughly speaking, the map $\zeta_{M|N}$ turns the matrix $T(u)$ up side down, so we take $\sd$ in the image space in order to keep the parities unchanged.

\begin{proposition}\label{zetadef} For all $1\leq i,j,h\leq \mu_a$, $1\leq k\leq \mu_{a+1}$, we have
\begin{eqnarray}
 \zeta_{M|N}\big(D_{a;i,j}(u)\big)&=&\overleftarrow{D}^{\prime}_{n+1-a;\mu_{a}+1-i,\mu_{a}+1-j}(u), \qquad\forall 1\leq a\leq n, \label{zd} \\
 \zeta_{M|N}\big(E_{a;h,k}(u)\big)&=&-\overleftarrow{F}_{n-a;\mu_{a}+1-h,\mu_{a+1}+1-k}(u), \quad\forall 1\leq a\leq n-1, \label{sze}\\
 \zeta_{M|N}\big(F_{a;k,h}(u)\big)&=&-\overleftarrow{E}_{n-a;\mu_{a+1}+1-k,\mu_{a}+1-h}(u), \quad\forall 1\leq a\leq n-1\label{szf}.
  \end{eqnarray}
\end{proposition}

The following proposition follows directly from (\ref{RTT}), (\ref{psid}), and Lemma \ref{corcommute}. In consequence, the relations among the $D$'s are obtained. One should notice that these $D$'s could be even or odd according to (\ref{pad}). This is different with the cases in \cite{Go,Pe1}, in which they are always even.
\begin{proposition}\label{dd0}
The relations among the elements
$\{D_{a;i,j}^{(r)},D_{a;i,j}^{\prime (r)}\}$ for all $r\geq 0$, ${1\leq i,j\leq \mu_a}$, $1\leq a\leq n$
are given by
\[
D_{a;i,j}^{(0)}=\delta_{ij},\]
\[
 {\displaystyle \sum_{t=0}^{r}D_{a;i,p}^{(t)}D_{a;p,j}^{\prime (r-t)}=\delta_{r0}\delta_{ij} },
\]
\begin{multline*}
[D_{a;i,j}^{(r)},D_{b;h,k}^{(s)}]=
    \delta_{ab}(-1)^{\pa{i}_a\pa{j}_a+\pa{i}_a\pa{h}_a+\pa{j}_a\pa{h}_a}\times\\
    \sum_{t=0}^{min(r,s)-1}\big(D_{a;h,j}^{(t)}D_{a;i,k}^{(r+s-1-t)}
     -D_{a;h,j}^{(r+s-1-t)}D_{a;i,k}^{(t)}\big).
\end{multline*}
\end{proposition}

It can be observed from the relations that the elements $\{D_{a;i,j}^{(r)},D_{a;i,j}^{\prime (r)}\}$ generate a subalgebra of $Y_{\mu}$, called the $Levi$ $subalgebra$ of $Y_{\mu}$, and denote it by $Y^0_{\mu}(\so)$. 
By Lemma \ref{corcommute}, we have  
\begin{eqnarray*}
Y^0_{\mu}(\so)&=&Y_{\mu_1}(\so_1)\psi_{\mu_1}(Y_{\mu_2}(\so_2))\psi_{\mu_1+\mu_2}(Y_{\mu_3}(\so_3))\cdots\psi_{\mu_1+\cdots+\mu_{n-1}}(Y_{\mu_n}(\so_n))\\
&\cong &Y_{\mu_1}(\so_1)\otimes Y_{\mu_2}(\so_2)\otimes\cdots \otimes Y_{\mu_n}(\so_n),
\end{eqnarray*}
where $\mu=(\mu_1,\mu_2,\ldots ,\mu_n)$ and $\so=\so_1\so_2\cdots\so_n$.
Note that in the special case when all $\mu_i=1$, the subalgebra $Y^0_{(1,\ldots,1)}(\so)$ is purely even and commutative. One may think $Y^0_{(1,\ldots,1)}(\so)$ in $Y_{M|N}$ as an analogue of the Cartan subalgebra  consisting of all diagonal matrices in $\gl_{M|N}$.

\section{Special Case: $n=2$}
In this section, we focus on the very first non-trivial case under our consideration; that is, $\mu=(\mu_1,\mu_2)$ with a fixed 01-sequence $\so=\so_1\so_2$. 
We list our parabolic generators as follow:
\begin{align*}
&\big\lbrace D_{a;i,j}^{(r)}, D_{a;i,j}^{\prime(r)} \,|\, {a=1,2;\; 1\leq i,j\leq \mu_a;\; r\geq 0}\big\rbrace,\\
&\big\lbrace E_{1;i,j}^{(r)} \,|\, 1\leq i\leq \mu_1, 1\leq j\leq\mu_{2};\; r\geq 1\big\rbrace,\\
&\big\lbrace F_{1;i,j}^{(r)} \,|\, 1\leq i\leq\mu_{2}, 1\leq j\leq \mu_1;\; r\geq 1\big\rbrace.
\end{align*}

The following proposition gives explicitly the relations among the generators other than those relations already obtained in Proposition \ref{dd0}.
{\allowdisplaybreaks
\begin{proposition}\label{n=2}
Let $\mu=(\mu_1,\mu_2)$ be a composition of $M+N$. The following identities hold in
$Y_{\mu}((u^{-1},v^{-1}))$:
\begin{align}
 (u-v)[D_{1;i,j}(u), E_{1;h,k}(v)]\label{p511}
        &=(-1)^{\pa{h}_1\pa{j}_1}\delta_{hj}\sum_{p=1}^{\mu_1}D_{1;i,p}(u)\big(E_{1;p,k}(v)-E_{1;p,k}(u)\big),\\[2mm]
 (u-v)[E_{1;i,j}(u), D^{\prime}_{2;h,k}(v)]
        &=(-1)^{\pa{h}_2\pa{j}_2}\del_{hj} \sum_{q=1}^{\mu_2}\big(E_{1;i,q}(u)-E_{1;i,q}(v)\big)D^{\prime}_{2;q,k}(v),\label{p512}\\[2mm]        
 (u-v)[D_{2;i,j}(u), E_{1;h,k}(v)]
        &\notag=(-1)^{\pa{h}_1\pa{k}_2+\pa{h}_1\pa{j}_2+\pa{j}_2\pa{k}_2}\times\\
        &\qquad\qquad\qquad D_{2;i,k}(u)\big(E_{1;h,j}(u)-E_{1;h,j}(v)\big),\label{p513}\\[2mm]              
 (u-v)[D_{1;i,j}(u), F_{1;h,k}(v)]
        &\notag=(-1)^{\pa{i}_1\pa{j}_1+\pa{h}_2\pa{i}_1+\pa{h}_2\pa{j}_1}\delta_{ik}\times\\
        &\qquad\qquad\qquad \sum_{p=1}^{\mu_1}\big(F_{1;h,p}(u)-F_{1;h,p}(v)\big)D_{1;p,j}(u),\label{p514}\\[2mm]
(u-v)[F_{1;i,j}(u), D^{\prime}_{2;h,k}(v)]
        &\notag=(-1)^{\pa{h}_2\pa{i}_2+\pa{h}_2\pa{j}_1+\pa{j}_1\pa{k}_2}\del_{ik}\times\\
        &\qquad\qquad\qquad \sum_{q=1}^{\mu_2}D^{\prime}_{2;h,q}(v)\big(F_{1;q,j}(v)-F_{1;q,j}(u)\big),\label{p515}\\[2mm]          
(u-v)[D_{2;i,j}(u), F_{1;h,k}(v)]
        &\notag=(-1)^{\pa{h}_2\pa{k}_1+\pa{h}_2\pa{j}_2+\pa{j}_2\pa{k}_1}\times\\
        &\qquad\qquad\qquad \big(F_{1;i,k}(v)-F_{1;i,k}(u)\big)D_{2;h,j}(u),\label{p516}\\[2mm]       
(u-v)[E_{1;i,j}(u), F_{1;h,k}(v)]
        &\notag=(-1)^{\pa{h}_2\pa{i}_1+\pa{i}_1\pa{j}_2+\pa{h}_2\pa{j}_2}D_{2;h,j}(u) D^{\prime}_{1;i,k}(u)\\
        &\qquad -(-1)^{\pa{h}_2\pa{k}_1+\pa{j}_2\pa{k}_1+\pa{h}_2\pa{j}_2}
        D^{\prime}_{1;i,k}(v) D_{2;h,j}(v),\label{p517}\\[2mm]
(u-v)[E_{1;i,j}(u), E_{1;h,k}(v)]
        &\notag=(-1)^{\pa{h}_1\pa{j}_2+\pa{j}_2\pa{k}_2+\pa{h}_1\pa{k}_2}\times\\
        &\qquad \big(E_{1;i,k}(u)-E_{1;i,k}(v)\big)\big(E_{1;h,j}(u)-E_{1;h,j}(v)\big),\label{p518}\\[2mm]
(u-v)[F_{1;i,j}(u), F_{1;h,k}(v)]
        &\notag=(-1)^{\pa{i}_2\pa{j}_1+\pa{h}_2\pa{i}_2+\pa{h}_2\pa{j}_1}\times\\
        &\qquad \big(F_{1;h,j}(v)-F_{1;h,j}(u)\big)\big(F_{1;i,k}(u)-F_{1;i,k}(v)\big).\label{p519} 
        \end{align}
The identities hold for all $1\leq i,j\leq \mu_1$ if $D_{1;i.j}(u)$ appears on the left-hand side of the equation, 
for all $1\leq h,k\leq \mu_2$ if $D_{2;h,k}^\prime(u)$ appears on the left-hand side of the equation, 
for all $1\leq i^\prime \leq \mu_1$, $1\leq j^\prime \leq \mu_2$ if $E_{1;i^\prime,j^\prime}(u)$ appears on the left-hand side of the equation, and
for all $1\leq h^\prime\leq \mu_2$, $1\leq k^\prime\leq \mu_1$ if $F_{1;h^\prime,k^\prime}(u)$ appears on the left-hand side of the equation.         
\end{proposition}
}
\begin{proof}
Our approach is similar to those in \cite[Section 6]{BK1} and \cite[Section 5]{Pe1}.
We compute the matrix products (\ref{T=FDE}) and (\ref{Tp=FDE})
with respect to the composition $\mu=(\mu_1,\mu_2)$ and get the following identities.
\begin{eqnarray}
\label{511}t_{i,j}(u)&=&D_{1;i,j}(u), \qquad\qquad\qquad\quad\;\forall 1\le i,j\le \mu_1,\\
\label{512}t_{i,\mu_1+j}(u)&=&D_{1;i,p}E_{1;p,j}(u),\qquad\qquad\quad\;\;\forall 1\le i\le \mu_1, 1\le j\le \mu_2,\\
\label{513}t_{\mu_1+i,j}(u)&=&F_{1;i,p}(u)D_{1;p,j}(u),\qquad\quad\quad\;\forall 1\le i\le \mu_2, 1\le j\le \mu_1,\\
\label{514}t_{\mu_1+i,\mu_1+j}(u)&=&F_{1;i,p}(u)D_{1;p,q}(u)E_{1;q,j}(u)+D_{2;i,j}(u),\,\,\,\forall 1\le i,j\le \mu_2,\\
\label{515}t^{\prime}_{i,j}(u)&=&D^{\prime}_{1;i,j}(u)+E_{1;i,p'}(u)D^{\prime}_{2;p',q'}(u)F_{1;q',j}(u),\forall 1\le i,j\le \mu_1,\\
\label{516}t^{\prime}_{i,\mu_1+j}(u)&=&-E_{1;i,p'}(u)D^{\prime}_{2;p',j}(u),\qquad\quad\forall 1\le i\le \mu_1, 1\le j\le \mu_2,\\
\label{517}t^{\prime}_{\mu_1+i,j}(u)&=&-D^{\prime}_{2;i,p'}(u)F_{1;p',j}(u),\qquad\quad\,\forall 1\le i\le \mu_2, 1\le j\le \mu_1,\\
\label{518}t^{\prime}_{\mu_1+i,\mu_1+j}(u)&=&D^{\prime}_{2;i,j}(u),\qquad\qquad\qquad\quad\;\,\forall 1\le i,j\le \mu_2,
\end{eqnarray}
where the indices $p,q$ (respectively, $p',q'$) are summed over $1,\ldots,\mu_1$ (respectively, $1,\ldots,\mu_2$).

(\ref{p511})--(\ref{p513}) can be proved similar to \cite[Lemma~6.3]{BK1} and \cite[Proposition~5.1]{Pe1}, except for some issues about the sign factors that have to be carefully treated. (\ref{p514})--(\ref{p516}) and (\ref{p519}) follow from applying the map $\zeta_{M|N}$ to (\ref{p511})--(\ref{p513}) and (\ref{p518}) with suitable choices of indices.
We prove (\ref{p517}) and (\ref{p518}) in detail here as illustrating examples about some new phenomenons and how we deal with them.

To show (\ref{p517}), we need other identities. Computing the brackets in (\ref{p512}) and (\ref{p514}) by definition, we have
\begin{multline}\label{2ef1}
(u-v)E_{1;\alpha ,j}(u)D^{\prime}_{2;h,\beta}(v)
  =(-1)^{\pa{h}_2\pa{j}_2}\delta_{hj}\big(E_{1;\alpha ,q}(u)-E_{1;\alpha ,q}(v)\big)D^{\prime}_{2;q,\beta}(v)\\
+(-1)^{(\pa{\alpha}_1+\pa{j}_2)(\pa{h}_2+\pa{\beta}_2)}(u-v)D^{\prime}_{2;h,\beta}(v)E_{1;\alpha ,j}(u),
\end{multline}
\begin{multline}\label{2ef2}
(u-v)F_{1;\beta ,k}(v)D_{1;i,\alpha}(u)=
(-1)^{(\pa{i}_1+\pa{\alpha}_1)(\pa{\beta}_1+\pa{k}_1)}(u-v)D_{1;i,\alpha}(u)F_{1;\beta ,k}(v)\\
-(-1)^{\pa{i}_1\pa{k}_1}\delta_{ki}\big(F_{1;\beta ,p}(u)-F_{1;\beta ,p}(v)\big)D_{1;p,\alpha}(u),
\end{multline}
where $\alpha$, $p$ (respectively, $\beta$, $q$) are summed over $1,\ldots,\mu_1$ (respectively, $1,\ldots, \mu_2$).

By (\ref{usefull}), we have
\begin{multline*}
(u-v)[t_{i,\mu_1+j}(u),t^{\prime}_{\mu_1+h,k}(v)]=(-1)^{\pa{i}_1\pa{j}_2+\pa{i}_1\pa{h}_2+\pa{j}_2\pa{h}_2}\times \\
\big(\delta_{hj}\sum_{g=1}^{M+N}t_{ig}(u)t^{\prime}_{gk}(v)
-\delta_{ki}\sum_{s=1}^{M+N}t^{\prime}_{\mu_1+h,s}(v)t_{s,\mu_1+j}(u)\big).
\end{multline*}
Substituting by (\ref{511})$-$(\ref{518}), we may rewrite the above identity as the following
\begin{align}\notag
&D_{1;i,\alpha}(u)(u-v)E_{1;\alpha ,j}(u)D^{\prime}_{2;h,\beta}(v)F_{1;\beta,k}(v)+
(-1)^{\pa{j}_2\pa{h}_2}\delta_{hj}D_{1;i,\alpha}(u)D^{\prime}_{1;\alpha ,k}(v)\\\notag
&+(-1)^{\pa{j}_2\pa{h}_2}\delta_{hj}D_{1;i,\alpha}(u)\big(E_{1;\alpha ,q}(v)-E_{1;\alpha ,q}(u)\big)D^{\prime}_{2;q,\beta}(v) F_{\beta ,k}(v)\\\notag
&=(-1)^{(\pa{i}_1+\pa{j}_2)(\pa{h}_2+\pa{k}_1)}(u-v)D^{\prime}_{2;h,\beta}(v) F_{1;\beta ,k}(v)D_{1;i,\alpha}(u)E_{1;\alpha,j}(u)\\\notag
&+(-1)^{(\pa{i}_1+\pa{j}_2)(\pa{h}_2+\pa{k}_1)}\delta_{ki}D_{2;h,\beta}^\prime(v)\big(F_{1;\beta ,p}(u)-F_{1;\beta ,p}(v)\big)D_{1;p,\alpha}(u) E_{1;\alpha ,j}(u)\\ \label{2ef3}
&+(-1)^{(\pa{i}_1+\pa{j}_2)(\pa{h}_2+\pa{k}_1)}\delta_{ki}D^{\prime}_{2;h,\beta}(v)D_{2;\beta ,j}(u),
\end{align}
where $\alpha$, $p$ (respectively, $\beta$, $q$) are summed over $1,\ldots,\mu_1$(resp. $1,\ldots,\mu_2$). Substituting (\ref{2ef1}) and (\ref{2ef2}) into (\ref{2ef3}) and simplifying the result, we obtain
\begin{multline*}
D_{1;i,\alpha}(u)D^\prime_{2;h\beta}(v)D_{2;\beta,j}(v)D^{\prime}_{1;\alpha ,k}(v)\\
+(-1)^{\pa{\alpha}_1\pa{h}_2+\pa{\alpha}_1\pa{k}_1+\pa{j}_2\pa{\beta}_2}(u-v)D_{1;i,\alpha}(u) D^{\prime}_{2;h,\beta}(v)E_{1;\alpha ,j}(u)F_{1;\beta ,k}(v)\\
=
(-1)^{\pa{\alpha}_1\pa{h}_2+\pa{\alpha}_1\pa{k}_1+\pa{j}_2\pa{k}_1}(u-v)D_{1;i,\alpha}(u) D^{\prime}_{2;h,\beta}(v) F_{1;\beta ,k}(v)E_{1;\alpha ,j}(u)\\
+(-1)^{\pa{k}_1\pa{j}_2+\pa{h}_2\pa{\alpha}_1+\pa{\beta}_2\pa{\alpha}_1+\pa{\beta}_2\pa{k}_1}D_{1;i,\alpha}(u) D^{\prime}_{2;h,\beta}(v) D_{1;\alpha,k}^\prime(u) D_{2;\beta ,j}(u).
\end{multline*}
Note that in the above equality, the index $i$ is not involved in those sign factors. We may multiply the matrix $D_1^\prime(u)$ from the left to both sides of the equality above so that we have:

\begin{multline*}
D^\prime_{2;h\beta}(v)D_{2;\beta,j}(v)D^{\prime}_{1;i ,k}(v)\\
+(-1)^{\pa{i}_1\pa{h}_2+\pa{i}_1\pa{k}_1+\pa{j}_2\pa{\beta}_2}(u-v)D^{\prime}_{2;h,\beta}(v)E_{1;i ,j}(u)F_{1;\beta ,k}(v)\\
=
(-1)^{\pa{i}_1\pa{h}_2+\pa{i}_1\pa{k}_1+\pa{j}_2\pa{k}_1}(u-v) D^{\prime}_{2;h,\beta}(v) F_{1;\beta ,k}(v)E_{1;i ,j}(u)\\
+(-1)^{\pa{k}_1\pa{j}_2+\pa{h}_2\pa{i}_1+\pa{\beta}_2\pa{i}_1+\pa{\beta}_2\pa{k}_1}D^{\prime}_{2;h,\beta}(v) D_{1;i,k}^\prime(u) D_{2;\beta ,j}(u).
\end{multline*}
Similar to the above computation, we want to multiply $D_2(v)$ from the left to the above identity. However, we can {\em not} do this directly since the index $h$ is involved in some sign factors; such a phenomenon didn't appear in \cite{BK1, Go, Pe1}.
It turns out that we may multiply a suitable sign factor $(-1)^{\pa{i}_1\pa{h}_2}$ to the above identity so that
\begin{multline*}
(-1)^{\pa{i}_1\pa{j}_2}D^\prime_{2;h,\beta}(v)D_{2;\beta,j}(v)D^{\prime}_{1;i ,k}(v)\\
+(-1)^{\pa{i}_1\pa{k}_1+\pa{j}_2\pa{\beta}_2}(u-v)D^{\prime}_{2;h,\beta}(v)E_{1;i ,j}(u)F_{1;\beta ,k}(v)\\
=
(-1)^{\pa{i}_1\pa{k}_1+\pa{j}_2\pa{k}_1}(u-v) D^{\prime}_{2;h,\beta}(v) F_{1;\beta ,k}(v)E_{1;i ,j}(u)\\
+(-1)^{\pa{k}_1\pa{j}_2+\pa{\beta}_2\pa{i}_1+\pa{\beta}_2\pa{k}_1}D^{\prime}_{2;h,\beta}(v) D_{1;i,k}^\prime(u) D_{2;\beta ,j}(u).
\end{multline*}
Observe that in the very first term we have $D^\prime_{2;h\beta}(v)D_{2;\beta,j}(v)$, which is $\delta_{hj}$, so we may replace $\pa{h}_2$ by $\pa{j}_2$ in the sign factor. Now those sign factors in the above result are free from the index $h$ so we may multiply $D_2(v)$ from the left to obtain
\begin{multline*}
(-1)^{\pa{i}_1\pa{j}_2}D_{2;h,j}(v)D^{\prime}_{1;i ,k}(v)
+(-1)^{\pa{i}_1\pa{h}_1+\pa{j}_2\pa{h}_2}(u-v)E_{1;i ,j}(u)F_{1;h ,k}(v)\\
=
(-1)^{\pa{i}_1\pa{k}_1+\pa{j}_2\pa{k}_1}(u-v)  F_{1;h ,k}(v)E_{1;i ,j}(u)\\
+(-1)^{\pa{k}_1\pa{j}_2+\pa{h}_2\pa{i}_1+\pa{h}_2\pa{k}_1} D_{1;i,k}^\prime(u) D_{2;h ,j}(u).
\end{multline*}
Collecting the corresponding terms and using the fact that $D_{2;h,j}(u)$, $D_{1;i,k}^\prime(u)$ supercommute, we derive (\ref{p517}).
Such a technique appears very often in the remaining part of this article.

For (\ref{p518}), we start with $[t_{i,\mu_1+j}(u), t^{\prime}_{h,\mu_1+k}(v)]=0.$
Multiplying $(u-v)^2$ and computing the bracket after substituting by (\ref{512}) and (\ref{516}), we have
\begin{align}\notag
&(u-v)^2D_{1;i,p}(u)E_{1;p,j}(u)E_{1;h,q}(v)D^{\prime}_{2;q,k}(v)\qquad\qquad\qquad\qquad\qquad\qquad\qquad\\
&-(-1)^{\pa{p}_1\pa{h}_1+\pa{p}_1\pa{q}_2+\pa{j}_2\pa{q}_2}(u-v)E_{1;h,q}(v)D_{1;i,p}(u)(u-v)D^{\prime}_{2;q,k}(v)E_{1;p,j}(u)=0,\label{2ee1}
\end{align}
where the indices $p$ and $q$ are summed from 1 to $\mu_1$ and $\mu_2$, respectively.
Computing the brackets in (5.1) and (\ref{p512}), we obtain the following two identities
\begin{multline*}
(u-v)(-1)^{(\pa{i}_1+\pa{p}_1)(\pa{h}_1+\pa{q}_2)}E_{1;h,q}(v)D_{1;i,p}(u)\\
=(u-v)D_{1;i,p}(u)E_{1;h,q}(v)
-\delta_{hp}(-1)^{\pa{h}_1\pa{p}_1}D_{1;i,g_1}(u)\big(E_{1;g_1,q}(v)-E_{1;g_1,q}(u)\big),
\end{multline*}
\begin{multline*}
(u-v)(-1)^{(\pa{p}_1+\pa{j}_2)(\pa{q}_2+\pa{k}_2)}D^{\prime}_{2;q,k}(v)E_{1;p,j}(u)\\
=(u-v)E_{1;p,j}(u)D^{\prime}_{2;q,k}(v)
+ \delta_{jq}(-1)^{\pa{q}_2\pa{j}_2}\big(E_{1;p,g_2}(u)-E_{1;p,g_2}(v)\big)D^{\prime}_{2;g_2,k}(v),
\end{multline*}
where the indices $p, g_1$ (respectively, $q, g_2$) are summed from 1 to $\mu_1$ (respectively, $\mu_2$).
Substituting these two into the second term of (\ref{2ee1}), multiplying some suitable choices of sign factors as in the proof of (\ref{p517}) so that we may multiply $D_1(u)$ from the left and $D_2(v)$ from the right simultaneously, we derive that
\begin{multline}
(u-v)^2[E_{1;i,j}(u),E_{1;h,k}(v)]=\\
(-1)^{\pa{i}_1\pa{j}_2+\pa{i}_1\pa{h}_1+\pa{j}_2\pa{h}_1}(u-v)E_{1;h,j}(v)\big(E_{1;i,k}(v)-E_{1;i,k}(u)\big)\\
+(-1)^{\pa{j}_2\pa{h}_1+\pa{j}_2\pa{k}_2+\pa{h}_1\pa{k}_2}(u-v)\big(E_{1;i,k}(u)-E_{1;i,k}(v)\big)E_{;1h,j}(u)\\
+\big(E_{1;i,j}(v)-E_{1;i,j}(u)\big)\big(E_{1;h,k}(u)-E_{1;h,k}(v)\big).\label{2ee2}
\end{multline}
For a power series $P$ in $Y_{\mu}[[u^{-1}, v^{-1}]]$, we write $\big\lbrace P\big\rbrace_{d}$ for the homogeneous component of $P$ of total degree $d$ in the variables $u^{-1}$ and $v^{-1}$. Then (\ref{p518}) is a consequence of the following claim.

{\bf{Claim:}} For $d\ge 1$, we have
\begin{multline*}
(u-v)\big\lbrace [E_{1;i,j}(u), E_{1;h,k}(v)]\big\rbrace_{d+1}=\\
\big\lbrace (-1)^{\pa{j}_2\pa{h}_1+\pa{j}_2\pa{k}_2+\pa{h}_1\pa{k}_2}\big(E_{1;i,k}(u)-E_{1;i,k}(v)\big)\big(E_{1;h,j}(u)-E_{1;h,j}(v)\big)\big\rbrace_d
\end{multline*}
We prove the claim by induction on $d$.
For $d=1$, take $\big\lbrace\:\big\rbrace_0$ on (\ref{2ee2}) so that 
\[
(u-v)^2\big\lbrace [E_{1;i,j}(u),E_{1;h,k}(v)]\big\rbrace_2=0.
\]
Note that the right-hand side of (\ref{2ee2}) is zero when $u=v$, hence we may divide both sides by $(u-v)$ and therefore $(u-v)\big\lbrace [E_{1;i,j}(u),E_{1;h,k}(v)]\big\rbrace_2=0$, as desired.
\\
Assuming that the claim is true for some $d>1$, so we have
\begin{multline}
(-1)^{\pa{k}_2\pa{j}_2+\pa{k}_2\pa{h}_1+\pa{j}_2\pa{h}_1+1}\big\lbrace[E_{1;i,k}(u), E_{1;h,j}(v)]\big\rbrace_{d+1}=\\
\Big\lbrace \frac{\big(E_{1;i,j}(v)-E_{1;i,j}(u)\big)\big(E_{1;h,k}(u)-E_{1;h,k}(v)\big)}{u-v}\Big\rbrace_d\label{2ee3}
\end{multline}
Note that the right-hand side of (\ref{2ee3}) is zero when $u=v$, which implies
\begin{equation}
E_{1;i,k}(v)E_{1;h,j}(v)=(-1)^{\pa{i}_1\pa{h}_1+\pa{j}_2\pa{h}_1+\pa{i}_1\pa{j}_2+1}E_{1;h,j}(v)E_{1;i,k}(v).\label{2ee4}
\end{equation}
Take $\big\lbrace\:\big\rbrace_d$ on (\ref{2ee2}) and replace the last term by (\ref{2ee3}):
\begin{align*}
(&u-v)^2\big\lbrace [E_{1;i,j}(u), E_{1;h,k}(v)]\big\rbrace_{d+2}\\
&=(u-v)(-1)^{\pa{i}_1\pa{j}_2+\pa{i}_1\pa{h}_1+\pa{j}_2\pa{h}_1}\big\lbrace E_{1;h,j}(v)\big(E_{1;i,k}(v)-E_{1;i,k}(u)\big)\big\rbrace_{d+1}\\
& +(u-v)(-1)^{\pa{j}_2\pa{h}_1+\pa{j}_2\pa{k}_2+\pa{h}_1\pa{k}_2}\big\lbrace \big(E_{1;i,k}(u)-E_{1;i,k}(v)\big)E_{1;h,j}(u)\big\rbrace_{d+1}\\
& +(u-v)\big\lbrace (-1)^{\pa{k}_2\pa{j}_2+\pa{k}_2\pa{h}_1+\pa{j}_2\pa{h}_1}E_{1;i,k}(u)E_{1;h,j}(v)-\\
&\qquad\qquad\qquad\qquad\qquad\qquad 
(-1)^{\pa{i}_1\pa{h}_1+\pa{j}_2\pa{h}_1+\pa{i}_1\pa{j}_2+1}E_{1;h,j}(v)E_{1;i,k}(u)\big\rbrace_d\\
&=(u-v)\big\lbrace (-1)^{\pa{i}_1\pa{j}_2+\pa{i}_1\pa{h}_1+\pa{j}_2\pa{h}_1}E_{1;h,j}(v) E_{1;i,k}(v)\big\rbrace_{d+1}\\
& +(u-v)(-1)^{\pa{j}_2\pa{h}_1+\pa{j}_2\pa{k}_2+\pa{h}_1\pa{k}_2}\big\lbrace E_{1;i,k}(u)E_{1;h,j}(u)-E_{1;i,k}(v)E_{1;h,j}(u) \big\rbrace_{d+1}\\
& +(u-v)\big\lbrace (-1)^{\pa{k}_2\pa{j}_2+\pa{k}_2\pa{h}_1+\pa{j}_2\pa{h}_1}E_{1;i,k}(u)E_{1;h,j}(v)\big\rbrace_d
\end{align*}
Substituting the term $(-1)^{\pa{i}_1\pa{j}_2+\pa{i}_1\pa{h}_1+\pa{j}_2\pa{h}_1}E_{1;h,j}(v) E_{1;i,k}(v)$ by (\ref{2ee4}) and simplifying the result, we have
\begin{align*}
&(u-v)^2\big\lbrace [E_{1;i,j}(u), E_{1;h,k}(v)]\big\rbrace_{d+2}=\\
&(u-v)\big\lbrace (-1)^{\pa{j}_2\pa{h}_1+\pa{j}_2\pa{k}_2+\pa{h}_1\pa{k}_2} \big(E_{1;i,k}(u)-E_{1;i,k}(v)\big)\big(E_{1;h,j}(u)-E_{1;h,j}(v)\big)\big\rbrace_{d+1}.
\end{align*}
Dividing both sides by $u-v$ establishes the claim.
\end{proof}

\section{Special Cases: $n=3$ and the super Serre relations}
In this section, we will consider the generators $D$'s, $E$'s and $F$'s in different super Yangians at the same time but using the same notation. It should be clear from the context which super Yangian we are dealing with. 

Similar to the proof of Proposition \ref{n=2}, we compute the matrix products
(\ref{T=FDE}) and (\ref{Tp=FDE}) with respect to the composition $\mu=(\mu_1,\mu_2,\mu_3)$ and derive the following identities.
\begin{eqnarray}
\label{610}t_{i,j}(u)&=&D_{1;i,j},\\
\label{611}t_{i,\mu_1+j}(u)&=&\sum_{p=1}^{\mu_1}D_{1;i,p}E_{1;p,j}(u),\\
\label{612}t_{i,\mu_1+\mu_2+j}(u)&=&\sum_{p=1}^{\mu_1}D_{1;i,p}E_{1,3;p,j}(u),\\
\label{613}t^{\prime}_{i,\mu_1+\mu_2+j}(u)&=&\sum_{p^\prime=1}^{\mu_3}\sum_{q=1}^{\mu_2}\big(E_{1;i,q}(u)E_{2;q,p^\prime}(u)-E_{1,3;i,p^\prime}(u)\big)D^{\prime}_{3;p^\prime,j}(u),\\
\label{614}t^{\prime}_{\mu_1+i,\mu_1+\mu_2+j}(u)&=&-\sum_{p^\prime=1}^{\mu_3}E_{2;i,p'}(u)D^{\prime}_{3;p^\prime,j}(u),\\
\label{615}t^{\prime}_{\mu_1+\mu_2+i,\mu_1+j}(u)&=&-\sum_{p^\prime=1}^{\mu_3}D^{\prime}_{3;i,p^\prime}(u)F_{2;p^\prime,j}(u),
\end{eqnarray}
where (\ref{610}) holds for all $1\leq i,j\leq \mu_1$, 
(\ref{611}) holds for all $1\leq i\leq \mu_1$, $1\leq j\leq \mu_2$,
(\ref{612}) and (\ref{613}) hold for all $1\leq i\leq \mu_1$, $1\leq j\leq \mu_3$,
(\ref{614}) holds for all $1\leq i\leq \mu_2$, $1\leq j\leq \mu_3$,
and (\ref{615}) holds for all $1\leq i\leq \mu_3$, $1\leq j\leq \mu_2$.

\begin{lemma}\label{3be}
The following identities hold in $Y_{(\mu_1,\mu_2,\mu_3)}((u^{-1},v^{-1}))$:
\begin{equation}\label{61a}
[E_{1;i,j}(u), F_{2;h,k}(v)] = 0,
\end{equation}
\begin{multline}\label{61b}
[E_{1;i,j}(u), E_{2;h,k}(v)] = \\
  \dfrac{(-1)^{\pa{j}_2\pa{h}_2}}{u-v}\delta_{hj}\sum_{q=1}^{\mu_2}\big\lbrace\big(E_{1;i,q}(u)-E_{1;i,q}(v)\big)E_{2;q,k}(v)
+ E_{1,3;i,k}(v) - E_{1,3;i,k}(u)\big\rbrace,
\end{multline}
\begin{multline}\label{61c}
[E_{1,3;i,j}(u), E_{2;h,k}(v)] =\\
(-1)^{\pa{i}_1\pa{j}_3+\pa{i}_1\pa{h}_2+\pa{h}_2\pa{j}_3+\pa{g}_2}
E_{2;h,j}(v) [E_{1;i,g}(u), E_{2;g,k}(v)],
\end{multline}
\begin{multline}\label{61d}
\sum_{q=1}^{\mu_2}[E_{1;i,j}(u), E_{1,3;h,k}(v) - E_{1;h,q}(v) E_{2;q,k}(v)] =\\
(-1)^{\pa{h}_1\pa{j}_2+\pa{j}_2\pa{k}_3+\pa{h}_1\pa{k}_3+\pa{g}_2+1}
[E_{1;i,g}(u), E_{2;g,k}(v)] E_{1;h,j}(u).
\end{multline}
Here (\ref{61a}) holds for all $1\leq i\leq \mu_1$, $1\leq j,k\leq \mu_2$, $1\leq h\leq \mu_3$, (\ref{61b}) holds for all $1\leq i\leq \mu_1$, $1\leq j,h\leq \mu_2$, $1\leq k\leq \mu_3$, (\ref{61c}) holds for all $1\leq i\leq \mu_1$, $1\leq h\leq \mu_2$, $1\leq j,k\leq \mu_3$, and (\ref{61d}) holds for all $1\leq i,h\leq \mu_1$, $1\leq j,g\leq \mu_2$, $1\leq k\leq \mu_3$. 
\end{lemma}

\begin{proof}

By (\ref{usefull}), we have $[t_{i,\mu_1+j}(u),t^{\prime}_{\mu_1+\mu_2+h,\mu_1+k}(v)]=0$. Substituting by (\ref{611}) and (\ref{615}), we have
\[
[D_{1;i,p}(u)E_{1;p,j}(u),-D^{\prime}_{3;h,q}(v)F_{2;q,k}(v)]=0.
\]
Computing the bracket, we obtain
\begin{multline}\label{3ee1}
D_{1;i,p}(u)E_{1;p,j}(u)D^{\prime}_{3;h,q}(v)F_{2;q,k}(v)-\\
(-1)^{(\pa{i}_1+\pa{j}_2)(\pa{h}_3+\pa{k}_2)}D^{\prime}_{3;h,q}(v)F_{2;q,k}(v)D_{1;i,p}(u)E_{1;p,j}(u)=0,
\end{multline}
where $p$ and $q$ are summed over $1,\ldots,\mu_1$ and $1,\ldots,\mu_3$, respectively. Similarly, by (\ref{usefull}), we have
\[
[t_{ij}(u),t^{\prime}_{\mu_1+\mu_2+h,\mu_1+k}(v)]=[t_{i,\mu_1+j}(u),t^{\prime}_{\mu_1+\mu_2+h,\mu_1+\mu_2+k}(v)]=0,
\]
which implies that
\[
[D_{1;i,j}(u),F_{2;h,k}(v)]=[E_{1;i,j}(u),D^{\prime}_{3;h,k}(v)]=0.
\]
Substituting these into (\ref{3ee1}) and using the fact that $D_{1;i,j}(u)$, $D^{\prime}_{3;h,k}(v)$ supercommute, we have
\begin{multline*}
(-1)^{\pa{j}_2\pa{q}_3}D_{1;i,p}(u)D^{\prime}_{3;h,q}(v)E_{1;p,j}(u)F_{2;q,k}(v)-\\
(-1)^{\pa{j}_2\pa{k}_2+\pa{p}_1\pa{q}_3+\pa{p}_1\pa{k}_2}D_{1;i,p}(u)D^{\prime}_{3;h,q}(v)F_{2;q,k}(v)E_{1;p,j}(u)=0.
\end{multline*}
The sign factors are free from the indices $i$ and $h$. Multiplying $D_{3}(v)D^{\prime}_{1}(u)$ from the left, we obtain (\ref{61a})

By (\ref{usefull}) again, we have
\[
(u-v)[t_{i,\mu_1+j}(u),t^{\prime}_{\mu_1+h,\mu_1+\mu_2+k}(v)]=(-1)^{\pa{j}_2\pa{h}_2}\delta_{jh}\sum_{s=1}^{M+N}t_{is}(u)t_{s,\mu_1+\mu_2+k}(v).
\]
Substituting by (\ref{610})--(\ref{615}), we have
\begin{multline}\label{3ee2}
(u-v)[D_{1;i,p}(u)E_{1;p,j}(u),-E_{2;h,q}(v)D^{\prime}_{3;q,k}(v)]=\\
(-1)^{\pa{j}_2\pa{h}_2}\delta_{jh}D_{1;i,p}(u)\big\lbrace \big(E_{1;p,r}(v)E_{2;r,q}(v)-E_{1,3;p,q}(v)\big)\\
-E_{1;p,r}(u)E_{2;r,q}(v)+E_{1,3;p,q}(u)\big\rbrace D^{\prime}_{3;q,k}(v),
\end{multline}
where the indices $p,q,r$ are summed from 1 to $\mu_1,\mu_3,\mu_2$, respectively.
Using the facts that
\begin{eqnarray*}
\big[E_{1;i,j}(v),D^{\prime}_{3;h,k}(u)\big]=0,& \qquad\big(\text{proved in the proof of (\ref{61a})}\big)\\
\big[E_{2;i,j}(v),D_{1;h,k}(u)\big]=0,& \qquad \big(\text{obtained from}\;\;[t_{i,j}(u),t^{\prime}_{\mu_1+h,\mu_1+\mu_2+k}(v)]=0\big)
\end{eqnarray*}
we may rewrite (\ref{3ee2}) as the following
\begin{multline}\label{3ee3}
(u-v) D_{1;i,p}(u) \lbrace E_{1;p,j}(u)E_{2;h,q}(v)-\\
(-1)^{(\pa{j}_2+\pa{p}_1)(\pa{h}_2+\pa{q}_3)}E_{2;h,q}(v)E_{1;p,j}(u)\rbrace D^{\prime}_{3;q,k}(v)\\
=(-1)^{\pa{j}_2\pa{h}_2+1}\delta_{jh}D_{1;i,p}(u)\big\lbrace \big(E_{1;p,r}(v)E_{2;r,q}(v)-E_{1,3;p,q}(v)\big)\\
-E_{1;p,r}(u)E_{2;r,q}(v)+E_{1,3;p,q}(u)\big\rbrace D^{\prime}_{3;q,k}(v).
\end{multline}
The sign factors are free from the indices $i$ and $k$. Canceling $D_{1}(u)$ from the left and $D^{\prime}_3(v)$ from the right on both sides of (\ref{3ee3}) and dividing both sides by $u-v$, we have deduced (\ref{61b}).

To show (\ref{61c}), the identity (\ref{p512}) in $Y_{(\mu_2,\mu_3)}((u^{-1},v^{-1}))$ reads as
\begin{equation*}
(u-v)[E_{1;h,k}(u),D^{\prime}_{2;i,j}(v)]=(-1)^{\pa{i}_2\pa{k}_2}\delta_{ki}\big(E_{1;h,p}(u)-E_{1;h,p}(v)\big)D^{\prime}_{2;p,j}(v).
\end{equation*}
Applying the map $\psi_{\mu_1}$ to this identity and using (4.3)$-$(4.5), we have the following identity in $Y_{(\mu_1,\mu_2,\mu_3)}((u^{-1},v^{-1}))$
\begin{equation*}
(u-v)[E_{2;h,k}(u),D^{\prime}_{3;i,j}(v)]=(-1)^{\pa{i}_3\pa{k}_3}\delta_{ki}\big(E_{2;h,p}(u)-E_{2;h,p}(v)\big)D^{\prime}_{3;p,j}(v).
\end{equation*}
Taking the coefficient of $u^0$ in the above identity, we obtain
\begin{equation*}
[E_{2;h,k}^{(1)},D^{\prime}_{3;i,j}(v)]=-(-1)^{\pa{i}_3\pa{k}_3}\delta_{ki}E_{2;h,p}(v)D^{\prime}_{3;p,j}(v).
\end{equation*}
Also by (\ref{ter}), we have 
\begin{equation*}
E_{1,3;i,j}(u)=(-1)^{\pa{g}_2}[E_{1;i,g}(u),E^{(1)}_{2;g,j}],\; \text{for any} \; 1\leq g\leq\mu_2.
\end{equation*}
By (\ref{612}), (\ref{613}), the super Jacobi identity and the fact that $[E_{1;i,g}(u),D^{\prime}_{3;h,k}(v)]=0$, we have
\begin{align}\notag\label{3ee4}
[E_{1,3;i,j}(u),D^{\prime}_{3;h,k}(v)]&=\big[(-1)^{\pa{g}_2}[E_{1;i,g}(u),E^{(1)}_{2;g,j}],D^{\prime}_{3;h,k}(v)\big]\\\notag
&=(-1)^{\pa{g}_2}\big[E_{1;i,g}(u),[E^{(1)}_{2;g,j},D^{\prime}_{3;h,k}(v)]\big]+0\\\notag
&=(-1)^{\pa{g}_2}\big[E_{1;i,g}(u),-(-1)^{\pa{j}_3\pa{h}_3}\delta_{hj}E_{2;g,p}(v)D^{\prime}_{3;p,k}(v)\big]\\
&=(-1)^{1+\pa{g}_2+\pa{j}_3\pa{h}_3}\delta_{hj}\big[E_{1;i,g}(u),E_{2;g,p}(v)\big]D^{\prime}_{3;p,k}(v). 
\end{align}
By (\ref{usefull}) and (\ref{610})--(\ref{615}), we have
\[
[t_{i,\mu_1+\mu_2+j}(u),t^{\prime}_{\mu_1+h,\mu_1+\mu_2+k}(v)]=[D_{1;i,p}(u)E_{1,3;p,j}(u),-E_{2;h,q}(v)D^{\prime}_{3;q,k}(v)]=0,
\]
where $p$ and $q$ are summed from 1 to $\mu_1$ and $\mu_3$, respectively. Multiplying $D_1^{\prime}(u)$ from the left, we have $[E_{1,3;i,j}(u),E_{2;h,q}(v)D^{\prime}_{3;q,k}(v)]=0,$ which may be written as
\begin{multline*}
[E_{1,3;i,j}(u),E_{2;h,q}(v)]D^{\prime}_{3;q,k}(v)+\\
(-1)^{(\pa{i}_1+\pa{j}_3)(\pa{h}_2+\pa{q}_3)}E_{2;h,q}(v)[E_{1,3;i,j}(u),D^{\prime}_{3;q,k}(v)]=0.
\end{multline*}
Substituting the last bracket by (\ref{3ee4}), we have
\begin{multline*}
[E_{1,3;i,j}(u),E_{2;h,q}(v)]D^{\prime}_{3;q,k}(v)=\\
(-1)^{(\pa{i}_1+\pa{j}_3)(\pa{h}_2+\pa{j}_3)+\pa{g}_2+\pa{j}_3\pa{j}_3}E_{2;h,j}(v)[E_{1;i,g}(u),E_{2;g,p}(v)]D^{\prime}_{3;p,k}(v).
\end{multline*}
Multiplying $D_{3}(v)$ from the right to the above equality, we acquire (\ref{61c}).

Taking the coefficient of $u^0$ in (\ref{61b}), we have
\begin{equation}\label{3ee5}
[E_{1;i,j}^{(1)},E_{2;h,k}(v)]=(-1)^{\pa{j}_2\pa{h}_2}\delta_{hj}\big(-E_{1;i,q}(v)E_{2;q,k}(v)+E_{1,3;i,k}(v)\big).
\end{equation}
Taking the coefficient of $v^0$ in (\ref{p511}), we have
\begin{equation}\label{3ee6}
[D_{1;i,j}(u),E_{1;h,k}^{(1)}]=(-1)^{\pa{j}_1\pa{h}_1}\delta_{hj}D_{1;i,p}(u)E_{1;p,k}(u).
\end{equation}
Together with the super Jacobi identity and the fact that $[D_{1;i,j}(u),E_{2;g,k}(v)]=0$, (\ref{3ee5}) and (\ref{3ee6}) imply that
\begin{align*}
[D_{1;i,j}(u),E_{1,3;h,k}(v)&-E_{1;h,q}(v)E_{2;q,k}(v)]=[D_{1;i,j}(u),(-1)^{\pa{g_2}}\big[E_{1;h,g}^{(1)},E_{2;g,k}(v)]\big]\notag\\
&=(-1)^{\pa{g_2}}\big[[D_{1;i,j}(u),E_{1;h,g}^{(1)}],E_{2;g,k}(v)\big]+0\\
&=(-1)^{\pa{g_2}}[(-1)^{\pa{j}_1\pa{h}_1}\delta_{hj}D_{1;i,p}(u)E_{1;p,g}(u),E_{2;g,k}(v)]\\
&=(-1)^{\pa{g_2}+\pa{j}_1\pa{h}_1}\delta_{hj}D_{1;i,p}(u)[E_{1;p,g}(u),E_{2;g,k}(v)].
\end{align*}
Summing $j$ from 1 to $\mu_1$ in the above identity, we derive
\begin{multline}\label{3ee7}
(-1)^{\pa{i}_1(\pa{h}_1+\pa{k}_3)}\big(E_{1;h,r}(v)E_{2;r,k}(v)-E_{1,3;h,k}(v)\big)D_{1;i,p}(u)\\
-(-1)^{\pa{p}_1(\pa{h}_1+\pa{k}_3)}D_{1;i,p}(u)\big(E_{1;h,r}(v)E_{2;r,k}(v)-E_{1,3;h,k}(v)\big),\\
=(-1)^{\pa{g}_2+\pa{p}_1\pa{k}_3}\delta_{hp}D_{1;i,p'}(u)[E_{1;p',g}(u),E_{2;g,k}(v)]
\end{multline}
where $r,p,p'$ are summed over $\mu_2, \mu_1 ,\mu_1$, respectively.

On the other hand, by (\ref{usefull}) and (\ref{610})--(\ref{615}), we have
\begin{multline}\label{3ee8}
[t_{i,\mu_1+j}(u),t^{\prime}_{h,\mu_1+\mu_2+k}(v)]=\\
[D_{1;i,p}(u)E_{1;p,j}(u),\big(E_{1;h,r}(v)E_{2;r,q}(v)-E_{1,3;h,q}(v)\big)D^{\prime}_{3;q,k}(v)]=0,
\end{multline}
where $p$ and $q$ are summed from 1 to $\mu_1$ and $\mu_3$, respectively. Multiplying $D_{3}(v)$ from the right, note that $D_3$ supercommutes with $E_1$ and $D_1$, and computing the bracket, we obtain
\begin{multline}\label{3ee9}
D_{1;i,p}(u)E_{1;p,j}(u)\big(E_{1;h,r}(v)E_{2;r,k}(v)-E_{1,3;h,k}(v)\big)\\
-(-1)^{(\pa{i}_1+\pa{j}_2)(\pa{h}_1+\pa{k}_3)} \big(E_{1;h,r}(v)E_{2;r,k}(v)-E_{1,3;h,k}(v)\big)D_{1;i,p}(u)E_{1;p,j}(u) =0
\end{multline}
where $p$ and $r$ are summed from 1 to $\mu_1$ and $\mu_2$, respectively. Substituting (\ref{3ee7}) into the second term of (\ref{3ee9}), we have
\begin{multline*}
D_{1;i,p}(u)E_{1;p,j}(u)\big(E_{1;h,r}(v)E_{2;r,k}(v)-E_{1,3;h,k}(v)\big)\\
-(-1)^{(\pa{p}_1+\pa{j}_2)(\pa{h}_1+\pa{k}_3)} \big\lbrace D_{1;i,p}(u)\big(E_{1;h,r}(v)E_{2;r,k}(v)-E_{1,3;h,k}(v)\big)\big\rbrace E_{1;p,j}(u) \\
-(-1)^{(\pa{g}_2+\pa{j}_2)(\pa{h}_1+\pa{k}_3)} \big\lbrace D_{1;i,p}(u)\big[E_{1;p,g}(u),E_{2;g,k}(v)\big] \big\rbrace E_{1;h,j}(u)=0.
\end{multline*}
Multiplying $D^{\prime}_{1}(u)$ from the left, we deduce that
\begin{multline*}
E_{1;i,j}(u)\big(E_{1;h,q}(v)E_{2;q,k}(v)-E_{1,3;h,k}(v)\big)\\
-(-1)^{(\pa{i}_1+\pa{j}_2)(\pa{h}_1+\pa{k}_3)}\big(E_{1;h,q}(v)E_{2;q,k}(v)-E_{1,3;h,k}(v)\big)E_{1;i,j}(u)\\
=(-1)^{\pa{j}_2(\pa{h}_1+\pa{k}_3)+\pa{g}_2+\pa{h}_1\pa{k}_3}\big[E_{1;i,g}(u),E_{2;g,k}(v)\big]E_{1;h,j}(u).
\end{multline*}
Simplifying the above, we obtain (\ref{61d}).
\end{proof}

The following lemma can be deduced by applying the automorphism $\zeta_{N|M}$ to the corresponding identities of Lemma \ref{3be} in $Y_{\overleftarrow{\mu}}=Y_{(\mu_3,\mu_2,\mu_1)}=Y_{N|M}$.
\begin{lemma}\label{3bf}
The following identities hold in $Y_{(\mu_1,\mu_2,\mu_3)}((u^{-1},v^{-1}))$:
\begin{equation}\label{62a}
[F_{1;i,j}(u),E_{2;h,k}(v)]=0,
\end{equation}
\begin{multline}\label{62b}
[F_{1;i,j}(u),F_{2;h,k}(v)]=
\dfrac{(-1)^{\pa{i}_2\pa{j}_1+\pa{i}_2\pa{h}_3+\pa{j}_1\pa{h}_3}}{u-v}\delta_{ik}\times\\
\big\lbrace \sum_{q=1}^{\mu_2}F_{2;h,q}(v)\big(F_{1;q,j}(v)-F_{1;q,j}(u)\big)-F_{3,1;h,j}(v)+F_{3,1;h,j}(u)\big\rbrace,
\end{multline}
\begin{multline}\label{62c}
[F_{3,1;i,j}(u),F_{2;h,k}(v)]=
(-1)^{\pa{i}_3\pa{j}_1+\pa{i}_3\pa{h}_3+\pa{j}_1\pa{h}_3+\pa{g}_2+1}
[F_{2;h,g}(v),F_{1;g,j}(u)]F_{2;i,k}(v),
\end{multline}
\begin{multline}\label{62d}
\sum_{q=1}^{\mu_2}[F_{1;i,j}(u) , F_{2;h,q}(v)F_{1;q,k}(v)-F_{3,1;h,k}(v)]=\\
(-1)^{(\pa{h}_3+\pa{j}_1)(\pa{k}_1+\pa{g}_2)} F_{1;i,k}(u)[F_{1;g,j}(u),F_{2;h,g}(v)].
\end{multline}
Here (\ref{62a}) holds for all $1\leq i,h\leq \mu_2$, $1\leq j\leq \mu_1$, $1\leq k\leq \mu_3$, (\ref{62b}) holds for all $1\leq i,k\leq \mu_2$, $1\leq j\leq \mu_1$, $1\leq h\leq \mu_3$, (\ref{62c}) holds for all $1\leq i,h\leq \mu_3$, $1\leq j\leq \mu_1$, $1\leq j,k\leq \mu_2$, and (\ref{62d}) holds for all $1\leq i,g\leq \mu_2$, $1\leq j,k\leq \mu_1$, $1\leq h\leq \mu_3$.
\end{lemma}

Our next lemma is a generalization of \cite[Lemma 6.3]{Pe1}. It is surprising that there are no sign factors appearing in the resulting identities.
\begin{lemma}\label{3ef0}
The following identities hold in $Y_{(\mu_1,\mu_2,\mu_3)}[[u^{-1},v^{-1},w^{-1}]]$:
\begin{eqnarray}
\label{63a}&\big[[E_{1;i,j}(u),E_{2;h,k}(v)],E_{2;f,g}(v)\big]=0,\\[1mm]
\notag&\big[E_{1;i,j}(u),[E_{1;h,k}(u),E_{2;f,g}(v)]\big]=0,\\[1mm] 
\label{63c}&\big[[E_{1;i,j}(u),E_{2;h,k}(v)],E_{2;f,g}(w)\big]+
\big[[E_{1;i,j}(u),E_{2;h,k}(w)],E_{2;f,g}(v)\big]=0,\\[1mm]
\notag&\big[E_{1;i,j}(u),[E_{1;h,k}(v),E_{2;f,g}(w)]\big]+
\big[E_{1;i,j}(v),[E_{1;h,k}(u),E_{2;f,g}(w)]\big]=0,\\[1mm]
\notag&\big[[F_{1;i,j}(u),F_{2;h,k}(v)],F_{2;f,g}(v)\big]=0,\\[1mm]
\notag&\big[F_{1;i,j}(u),[F_{1;h,k}(u),F_{2;f,g}(v)]\big]=0,\\[1mm]
\notag&\big[[F_{1;i,j}(u),F_{2;h,k}(v)],F_{2;f,g}(w)\big]+\big[[F_{1;i,j}(u),F_{2;h,k}(w)],F_{2;f,g}(v)\big]=0,\\[1mm]
\notag&\big[F_{1;i,j}(u),[F_{1;h,k}(v),F_{2;f,g}(w)]\big]+\big[F_{1;i,j}(v),[F_{1;h,k}(u),F_{2;f,g}(w)]\big]=0.
\end{eqnarray}
The identities hold for all $1\leq i\leq \mu_a, 1\leq j\leq \mu_{a+1}$ if $E_{a;i.j}(u)$ appears and
hold for all $1\leq h \leq \mu_{a+1}$, $1\leq k \leq \mu_a$ if $F_{a;h,k}(u)$ appears, where $a$=1 or 2.
\end{lemma}

\begin{proof}
We prove (\ref{63a}) and (\ref{63c}) in detail while the others are similar.
To show (\ref{63a}), we first claim that
\[
[E_{a;i,j}(v),E_{a;h,k}(v)]=0 \quad \text{for}\; a=1,2 \quad\text{in}\quad Y_{(\mu_1,\mu_2,\mu_3)}[[v^{-1}]].
\]
Indeed, the case $a=1$ follows from (\ref{p518}) and $a=2$ follows from applying the map $\psi_{\mu_1}$ to (\ref{p518}) in $Y_{(\mu_2,\mu_3)}[[v^{-1}]]$.

By the super Jacobi identity, together with the above claim and (\ref{61b}), it suffices to prove the case when $j=h=f$. In this case, we compute the following bracket by Lemma~\ref{3be} as below.
{\allowdisplaybreaks
\begin{align*}
(&u-v)\big[[E_{1;i,j}(u),E_{2;j,k}(v)],E_{2;j,g}(v)\big]\\
&=(-1)^{(\pa{i}_1+\pa{j}_2)(\pa{j}_2+\pa{k}_3)}(u-v)\big[E_{2;j,k}(v),[E_{1;i,j}(u),E_{2;j,g}(v)]\big]\\
&=(-1)^{(\pa{i}_1+\pa{j}_2)(\pa{j}_2+\pa{k}_3)+(\pa{i}_1+\pa{g}_3)(\pa{j}_2+\pa{k}_3)}(u-v)\big[\,[E_{1;i,j}(u),E_{2;j,g}(v)] \,, E_{2;j,k}(v)\big]\\
&=(-1)^{(\pa{j}_2+\pa{g}_3)(\pa{j}_2+\pa{k}_3)}\big[(-1)^{\pa{j}_2} E_{1;i,q}(u) E_{2;q,g}(v)-E_{1;i,q}(v) E_{2;q,g}(v)\\
&\qquad\qquad + E_{1,3;i,g}(v) - E_{1,3;i,g}(u)\, , \,E_{2;j,k}(v)\big]\\
&=(-1)^{\pa{j}_2\pa{k}_3+\pa{j}_2\pa{g}_3+\pa{g}_3\pa{k}_3}\big\lbrace[E_{1;i,q}(u)E_{2;q,g}(v),E_{2;j,k}(v)]+[E_{1,3;i,g}(v),E_{2;j,k}(v)]\\
&\qquad -[E_{1;i,q}(v)E_{2;q,g}(v),E_{2;j,k}(v)]-[E_{1,3;i,g}(u),E_{2;j,k}(v)]\big\rbrace\\
&=(-1)^{\pa{j}_2\pa{k}_3+\pa{j}_2\pa{g}_3+\pa{g}_3\pa{k}_3}\big\lbrace
E_{1;i,q}(u)[E_{2;q,g}(v) , E_{2;j,k}(v)]  \\
&\,+(-1)^{(\pa{q}_2+\pa{g}_3)(\pa{j}_2+\pa{k}_3)}[ E_{1;i,q}(u), E_{2;j,k}(v)] E_{2;q,g}(v)
-E_{1;i,q}(v)[E_{2;q,g}(v) , E_{2;j,k}(v)]\\
&\,-(-1)^{(\pa{q}_2+\pa{g}_3)(\pa{j}_2+\pa{k}_3)}[ E_{1;i,q}(v), E_{2;j,k}(v)] E_{2;q,g}(v)\\
&\,+(-1)^{\pa{i}_1\pa{j}_2+\pa{i}_1\pa{g}_3+\pa{j}_2\pa{g}_3+\pa{j}_2} E_{2;j,g}(v) [E_{1;i,j}(v) , E_{2;j,k}(v)]\\
&\,-(-1)^{\pa{i}_1\pa{j}_2+\pa{i}_1\pa{g}_3+\pa{j}_2\pa{g}_3+\pa{j}_2} E_{2;j,g}(v) [E_{1;i,j}(u) , E_{2;j,k}(v)]\big\rbrace\\
&=(-1)^{\pa{j}_2\pa{k}_3+\pa{q}_2\pa{j}_2+\pa{q}_2\pa{k}_3} [E_{1;i,q}(u), E_{2;j,k}(v)] E_{2;q,g}(v) \\
&\,+(-1)^{\pa{j}_2+\pa{i}_1\pa{j}_2+\pa{i}_1\pa{g}_3+\pa{j}_2\pa{k}_3+\pa{g}_3\pa{k}_3} E_{2;j,g}(v)[ E_{1;i,j}(v), E_{2;j,k}(v)]\\
&\,-(-1)^{\pa{j}_2\pa{k}_3+\pa{q}_2\pa{j}_2+\pa{q}_2\pa{k}_3} [E_{1;i,q}(u), E_{2;j,k}(v)] E_{2;q,g}(v) \\
&\,-(-1)^{\pa{j}_2+\pa{i}_1\pa{j}_2+\pa{i}_1\pa{g}_3+\pa{j}_2\pa{k}_3+\pa{g}_3\pa{k}_3} E_{2;j,g}(v)[ E_{1;i,j}(u), E_{2;j,k}(v)]\\
&=(-1)^{\pa{j}_2}\big[ [E_{1;i,j}(u) , E_{2;j,k}(v)] , E_{2;j,g}(v) \big] 
-(-1)^{\pa{j}_2}\big[ [E_{1;i,j}(v) , E_{2;j,k}(v)] , E_{2;j,g}(v) \big]
\end{align*}
}
Thus we have
\begin{multline}\label{3eee0}
(u-v-(-1)^{\pa{j}_2})\big[[E_{1;i,j}(u),E_{2;j,k}(v)],E_{2;j,g}(v)\big]=\\
-(-1)^{\pa{j}_2}\big[[E_{1;i,j}(v),E_{2;j,k}(v)],E_{2;j,g}(v)\big]
\end{multline}
Note that the right-hand side of (\ref{3eee0}) is independent of the choice of $u$. Specifying $u=v+(-1)^{\pa{j}_2}$ in (\ref{3eee0}), we have 
\begin{equation*}
0=-(-1)^{\pa{j}_2}\big[[E_{1;i,j}(v),E_{2;j,k}(v)],E_{2;j,g}(v)\big],
\end{equation*}
and hence 
\begin{equation*}
(u-v-(-1)^{\pa{j}_2})\big[[E_{1;i,j}(u),E_{2;j,k}(v)],E_{2;j,g}(v)\big]=0, \,\,\text{for any } u.
\end{equation*}
Choose $u$ such that $u-v-(-1)^{\pa{j}_2}$ is invertible, and then (\ref{63a}) follows.

To show (\ref{63c}), it suffices to show that
\begin{equation}\label{3eeec1}
(u-w)(v-w)(u-v)\big[[E_{1;i,j}(u),E_{2;h,k}(v)],E_{2;f,g}(w)\big]
\end{equation}
is symmetric in $v$ and $w$. We may further assume $j=h$, as in the proof of (\ref{63a}).
By (\ref{61b}), we have
\begin{multline*}
(u-w)(v-w)(u-v)\big[[E_{1;i,j}(u),E_{2;h,k}(v)],E_{2;f,g}(w)\big]
=(v-w)(u-w)\times\\
(-1)^{\pa{j}_2}\big[ \big(E_{1;i,q}(u)-E_{1;i,q}(v) \big)E_{2;q,k}(v)+E_{1,3;i,k}(v)-E_{1,3;i,k}(u)\,,\,E_{2;f,g}(w)\big].
\end{multline*}
Computing the brackets by Lemma \ref{3be}, we have
{\allowdisplaybreaks
\begin{align*}\notag
(u&-w)(v-w)(u-v)\big[[E_{1;i,j}(u),E_{2;h,k}(v)],E_{2;f,g}(w)\big]\\ 
=&(v-w)(u-w)(-1)^{\pa{j}_2}\big[ [E_{1;i,q}(u) , E_{2;q,k}(v)] , E_{2;f,g}(w) \big]\\
&-(v-w)(u-w)(-1)^{\pa{j}_2}\big[ [E_{1;i,q}(v) , E_{2;q,k}(v)] , E_{2;f,g}(w) \big]\\
&+(v-w)(u-w)(-1)^{\pa{j}_2}\big[ E_{1,3;i,k}(v) ,  E_{2;f,g}(w) \big]\\
&-(v-w)(u-w)(-1)^{\pa{j}_2}\big[ E_{1,3;i,k}(u) , E_{2;f,g}(w) \big]\\
=&(v-w)(u-w)(-1)^{\pa{j}_2}\big[ [E_{1;i,q}(u) , E_{2;q,k}(v)] , E_{2;f,g}(w) \big]\\
&-(v-w)(u-w)(-1)^{\pa{j}_2}\big[ [E_{1;i,q}(v) , E_{2;q,k}(v)] , E_{2;f,g}(w) \big]\\
&+(v-w)(u-w)(-1)^{\pa{j}_2+\pa{\ell}_2+\pa{i}_1\pa{f}_2+\pa{i}_1\pa{k}_3+\pa{f}_2\pa{k}_3} 
E_{2;f,k}(w)\big[ E_{1;i,\ell}(v) ,  E_{2;\ell,g}(w) \big]\\
&-(v-w)(u-w)(-1)^{\pa{j}_2+\pa{\ell}_2+\pa{i}_1\pa{f}_2+\pa{i}_1\pa{k}_3+\pa{f}_2\pa{k}_3}E_{2;f,k}(w)\big[ E_{1;i,\ell}(u) ,  E_{2;\ell,g}(w) \big]\\
=&(u-w)(v-w)(-1)^{\pa{j}_2}\big\lbrace E_{1;i,q}(u)E_{2;q,k}(v)E_{2;f,g}(w)\\
&\qquad\qquad\qquad-(-1)^{(\pa{f}_2+\pa{g}_3)(\pa{i}_1+\pa{k}_3)}E_{2;f,g}(w)E_{1;i,q}(u)E_{2;q,k}(v)\big\rbrace\\
&-(u-w)(v-w)(-1)^{\pa{j}_2}\big\lbrace E_{1;i,q}(u)E_{2;q,k}(v)E_{2;f,g}(w) \\
&\qquad\qquad\qquad-(-1)^{(\pa{f}_2+\pa{g}_3)(\pa{i}_1+\pa{k}_3)}E_{2;f,g}(w)E_{1;i,q}(u)E_{2;q,k}(v)\big\rbrace\\
&+(v-w)(u-w)(-1)^{\pa{j}_2+\pa{\ell}_2+\pa{i}_1\pa{f}_2+\pa{i}_1\pa{k}_3+\pa{f}_2\pa{k}_3} 
E_{2;f,k}(w)\big[ E_{1;i,\ell}(v) ,  E_{2;\ell,g}(w) \big]\\
&-(v-w)(u-w)(-1)^{\pa{j}_2+\pa{\ell}_2+\pa{i}_1\pa{f}_2+\pa{i}_1\pa{k}_3+\pa{f}_2\pa{k}_3}E_{2;f,k}(w)\big[ E_{1;i,\ell}(u) ,  E_{2;\ell,g}(w) \big]\\
=&(u-w)(v-w)(-1)^{\pa{j}_2} E_{1;i,q}(u)\big[E_{2;q,k}(v),E_{2;f,g}(w)\big]\\
&-(u-w)(v-w)(-1)^{\pa{j}_2} E_{1;i,q}(v)\big[E_{2;q,k}(v),E_{2;f,g}(w)\big]\\
&-(u-w)(v-w)(-1)^{\pa{j}_2+(\pa{f}_2+\pa{g}_3)(\pa{i}_1+\pa{k}_3)} \big[E_{2;f,g}(w),E_{1;i,q}(u)\big]E_{2;q,k}(v)\\ 
&+(u-w)(v-w)(-1)^{\pa{j}_2+(\pa{f}_2+\pa{g}_3)(\pa{i}_1+\pa{k}_3)} \big[E_{2;f,g}(w),E_{1;i,q}(v)\big]E_{2;q,k}(v)\\
&+(u-w)(v-w)(-1)^{\pa{j}_2+\pa{\ell}_2+\pa{g}_3\pa{f}_2+\pa{g}_3\pa{k}_3+\pa{f}_2\pa{k}_3}\big[E_{1;i,\ell}(v),E_{2;\ell,g}(w)\big]E_{2;f,k}(w)\\
&-(u-w)(v-w)(-1)^{\pa{j}_2+\pa{\ell}_2+\pa{g}_3\pa{f}_2+\pa{g}_3\pa{k}_3+\pa{f}_2\pa{k}_3}\big[E_{1;i,\ell}(u),E_{2;\ell,g}(w)\big]E_{2;f,k}(w).
\end{align*}}
We use (\ref{p518}) and Lemma \ref{3be} to compute these brackets, then (\ref{3eeec1}) equals to
\begin{align*}
&\quad \varepsilon(u-w)E_{1;i,q}(u)\big(E_{2;q,g}(v)-E_{2;q,g}(w)\big)\big(E_{2;f,k}(v)-E_{2;f,k}(w)\big)\\
\,&-\varepsilon(u-w) E_{1;i,q}(v)\big(E_{2;q,g}(v)-E_{2;q,g}(w)\big)\big(E_{2;f,k}(v)-E_{2;f,k}(w)\big)\\
\,&+\varepsilon(v-w) \big\lbrace \big(E_{1;i,q}(u)-E_{1;i,q}(w)\big)E_{2;q,g}(w)+E_{1,3;i,g}(w)-E_{1,3;i,g}(u)\big\rbrace E_{2;f,k}(v)\\
\,&-\varepsilon(u-w) \big\lbrace \big(E_{1;i,q}(v)-E_{1;i,q}(w)\big)E_{2;q,g}(w)+E_{1,3;i,g}(w)-E_{1,3;i,g}(v)\big\rbrace E_{2;q,k}(v)\\
\,&+\varepsilon(u-w)\big\lbrace \big(E_{1;i,q}(v)-E_{1;i,q}(w)\big)E_{2;q,g}(w)+E_{1,3;i,g}(w)-E_{1,3;i,g}(v)\big\rbrace E_{2;f,k}(w)\\
\,&-\varepsilon(v-w)\big\lbrace \big(E_{1;i,q}(u)-E_{1;i,q}(w)\big)E_{2;q,g}(w)+E_{1,3;i,g}(w)-E_{1,3;i,g}(u)\big\rbrace E_{2;f,k}(w),
\end{align*}
where the $\varepsilon$ is a sign factor given by $\varepsilon=(-1)^{\pa{j}_2+\pa{g}_3\pa{f}_2+\pa{g}_3\pa{k}_3+\pa{f}_2\pa{k}_3}$, and the index $q$ is summed over $1,\ldots,\mu_2$.

Opening the parentheses of the above identity, one may check that the resulting expression is indeed symmetric in $v$ and $w$. Therefore, (\ref{3eeec1}) is symmetric in $v$ and $w$ and hence (\ref{63c}) is established.
\end{proof}

Suppose now that $\mu=(\mu_1,\ldots,\mu_n)$ with $n\geq 4$. The next lemma is a generalization of \cite[Lemma 5]{Go} and \cite[Lemma 7.2]{Pe1}, in which the results were proved only for one specific index. 
Here we show that they in fact hold {\em everywhere} and we require some of them to obtain the desired defining relations; see (\ref{p715}), (\ref{p716}).
\begin{lemma}\label{extra}
Associated to $\mu=(\mu_1,\mu_2,\ldots\mu_n)$ with $n\geq 4$, we have the following identities in $Y_{\mu}$, called the super Serre relations:
\begin{equation}\label{eeee}
\big[\,[E_{a;i,f_1}^{(r)},E_{a+1;f_2,j}^{(1)}],[E_{a+1;h,g_1}^{(1)},E_{a+2;g_2,k}^{(s)}]\,\big]=0,
\end{equation}
\begin{equation}\label{ffff}
\big[\,[F_{a;f_1,i}^{(r)},F_{a+1;j,f_2}^{(1)}],[F_{a+1;g_1,h}^{(1)},F_{a+2;k,g_2}^{(s)}]\,\big]=0,\\[1mm]
\end{equation}
for all $1\leq a\leq n-3$ and all $1\leq i\leq \mu_a$, $1\leq f_1,f_2,h\leq \mu_{a+1}$, $1\leq g_1,g_2,j\leq \mu_{a+2}$, $1\leq k\leq \mu_{a+3}$.
\end{lemma}
\begin{proof}
It suffices to prove the following special case of (\ref{eeee}) when $n=4$, while the general cases and (\ref{ffff}) can be achieved by applying the maps $\psi$ and $\zeta_{M|N}$:
\begin{equation}\label{e4e}
\big[\,[E_{1;i,f_1}^{(r)},E_{2;f_2,j}^{(1)}]\,,\,[E_{2;h,g_1}^{(1)},E_{3;g_2,k}^{(s)}]\,\big]=0.
\end{equation}
We claim that for all $1\leq i\leq \mu_1$, $1\leq h\leq \mu_2$, $1\leq j\leq \mu_3$, $1\leq k\leq \mu_4$, we have\begin{equation}\label{exeee}
[E_{1,3;i,j}(u)\,,\,E_{2;h,q}(v)E_{3;q,k}(v)-E_{2,4;h,k}(v)\,]=0,
\end{equation}
where the index $q$ is summed over $1,\ldots,\mu_3$.

To prove (\ref{exeee}), we multiply the matrix equalities (\ref{T=FDE}) and (\ref{Tp=FDE}) associated to the composition $(\mu_1,\mu_2, \mu_3,\mu_4)$ and derive the following identities.
\begin{eqnarray}
E_{1,3;i,j}(u)&=&D^{\prime}_{1;i,p}(u)t_{p,\mu_1+\mu_2+j}(u),\notag\\
E_{2;h,q}(v)E_{3;q,k}(v)-E_{2,4;h,k}(v)&=&t^{\prime}_{\mu_1+h,\mu_1+\mu_2+\mu_3+r}(v)D_{4;r,k}(v)\notag,
\end{eqnarray}
for all $1\le i\le\mu_1$, $1\le j\le\mu_3$, $1\le h\le\mu_2$, $1\le k\le\mu_4$, and the indices $p$, $q$, $r$ are summed from 1 to $\mu_1$, $\mu_3$, $\mu_4$, respectively. Substituting these identities into the bracket in (\ref{exeee}) and setting a notation $n_a:=\mu_1+\mu_2+\ldots+\mu_a$ for short, we have
\begin{align*}
&[E_{1,3;i,j}(u),E_{2;h,q}(v)E_{3;q,k}(v)-E_{2,4;h,k}(v)]\\
&=[D^{\prime}_{1;i,p}(u)t_{p,n_2+j}(u),t^{\prime}_{\mu_1+h,n_3+r}(v)D_{4;r,k}(v)]\\
&=D^{\prime}_{1;i,p}(u)t_{p,n_2+j}(u)t^{\prime}_{\mu_1+h,n_3+r}(v)D_{4;r,k}(v)\\
&\qquad\qquad\qquad-(-1)^{(\pa{i}_1+\pa{j}_3)(\pa{h}_2+\pa{k}_4)}t^{\prime}_{\mu_1+h,n_3+r}(v)D_{4;r,k}(v)D^{\prime}_{1;i,p}(u)t_{p,n_2+j}(u)\\
&=D^{\prime}_{1;i,p}(u)t_{p,n_2+j}(u)t^{\prime}_{\mu_1+h,n_3+r}(v)D_{4;r,k}(v)\\
&\qquad\qquad\qquad-(-1)^{(\pa{h}_2+\pa{r}_4)(\pa{p}_1+\pa{j}_3)}D^{\prime}_{1;i,p}(u)t^{\prime}_{\mu_1+h,n_3+r}(v)t_{p,n_2+j}(u)D_{4;r,k}(v)\\
&=D^{\prime}_{1;i,p}(u)[t_{p,n_2+j}(u),
t^{\prime}_{\mu_1+h,n_3+r}(v)]D_{4;r,k}(v)=0,
\end{align*}
and (\ref{exeee}) follows.
Note that in the above computation we have used the facts that
\[
D_{1;i,j}(u)=t_{ij}(u) \qquad \text{and} \qquad D^{\prime}_{4;i,j}(u)=t^{\prime}_{n_3+i,n_3+j}(u),
\]
therefore 
$[D_{1;i,j}(u),t^{\prime}_{\mu_1+h,n_3+k}(v)]=0$
and $[D^{\prime}_{4;i,j}(u),t_{h,n_2+k}(v)]=0$ \;by (\ref{usefull}).

To show (\ref{e4e}), by (\ref{61b}), we may assume that $f_1=f_2=f$ and $g_1=g_2=g$. 
Computing the following bracket by (\ref{61b}), we have
\begin{align*}
(u&-v)(w-z)\big[\,[E_{1;i,f}(u),E_{2;f,j}(v)]\,,\,[E_{2;h,g}(w),E_{3;g,k}(z)]\,\big]\notag\\
&=\big[\,(-1)^{\pa{f}_2}E_{1;i,q}(u)E_{2;q,k}(v)-E_{1;i,q}(v)E_{2;q,k}(v)+E_{1,3;i,k}(v)-E_{1,3;i,k}(u),\notag\\
&\qquad (-1)^{\pa{g}_3}E_{2;h,p}(w)E_{3;p,k}(z)-E_{2;h,p}(z)E_{3;p,k}(z)+E_{2,4;h,k}(z)-E_{2,4;h,k}(w)\,\big].\notag
\end{align*}
Taking its coefficient of $u^{-r}z^{-s}v^0w^0$, we have
\[
(-1)^{\pa{f}_2+\pa{g}_3}\sum_{t=1}^{s-1}[E_{1,3;i,j}^{(r)}\,,\,E_{2;h,p}^{(s-t)}E_{3;p,k}^{(t)}-E_{2,4;h,k}^{(s)}],
\]
which equals to the coefficient of $u^{-r}z^{-s}$ in 
\[
(-1)^{\pa{f}_2+\pa{g}_3}[E_{1,3;i,j}(u)\,,\,E_{2;h,p}(z)E_{3;p,k}(z)-E_{2,4;h,k}(z)],
\] 
which is zero by (\ref{exeee}).
Finally, the coefficient of $u^{-r}z^{-s}v^0w^0$ in
\begin{equation*}
(u-v)(w-z)\big[\,[E_{1;i,f}(u),E_{2;f,j}(v)]\,,\,[E_{2;h,g}(w),E_{3;g,k}(z)]\,\big]
\end{equation*}
is precisely $-\big[\,[E_{1;i,f}^{(r)},E_{2;f,j}^{(1)}]\,,\,[E_{2;h,g}^{(1)},E_{3;g,k}^{(s)}]\,\big]$ and (\ref{e4e}) follows.
\end{proof}

\section{The general Case}

Recall that our goal is to obtain the defining relations of $Y_\mu(\so)=Y_{M|N}$ in terms of the parabolic generators $\lbrace D_{a;i,j}^{(r)}, D_{a;i,j}^{\prime'(r)}\rbrace$, $\lbrace E_{a;i,j}^{(r)}\rbrace$, and $\lbrace F_{a;i,j}^{(r)}\rbrace$ associated to an arbitrary fixed composition $\mu$ of $M+N$ and an arbitrary fixed $0^M1^N$-sequence $\so$. The following proposition summarizes the results that we have established earlier.
\begin{proposition}\label{srlns}
The following relations hold in $Y_\mu(\so)$:
\begin{eqnarray}
\label{p701}D_{a;i,j}^{(0)}&=&\delta_{ij}\,,\\
\label{p702}\sum_{p=1}^{\mu_a}\sum_{t=0}^{r}D_{a;i,p}^{(t)}D_{a;p,j}^{\prime (r-t)}&=&\delta_{r0}\delta_{ij}\,,\\
\big[D_{a;i,j}^{(r)},D_{b;h,k}^{(s)}\big]&=&
    \delta_{ab}(-1)^{\pa{i}_a\pa{j}_a+\pa{i}_a\pa{h}_a+\pa{j}_a\pa{h}_a}\times \notag\\
    &&\sum_{t=0}^{min(r,s)-1}\big(D_{a;h,j}^{(t)}D_{a;i,k}^{(r+s-1-t)}-D_{a;h,j}^{(r+s-1-t)}D_{a;i,k}^{(t)}\big),
\end{eqnarray}
{\allowdisplaybreaks
\begin{multline}\label{p704}
 [D_{a;i,j}^{(r)}, E_{b;h,k}^{(s)}]
        =\delta_{a,b}\delta_{hj}(-1)^{\pa{h}_a\pa{j}_a}\sum_{p=1}^{\mu_a}\sum_{t=0}^{r-1} D_{a;i,p}^{(t)} E_{b;p,k}^{(r+s-1-t)}\\
        -\delta_{a,b+1}(-1)^{\pa{h}_b\pa{k}_a+\pa{h}_b\pa{j}_a+\pa{j}_a\pa{k}_a} \sum_{t=0}^{r-1} D_{a;i,k}^{(t)} E_{b;h,j}^{(r+s-1-t)},
        \end{multline}
\begin{multline}\label{p705}
 [D_{a;i,j}^{(r)}, F_{b;h,k}^{(s)}]
        =\delta_{a,b}(-1)^{\pa{i}_a\pa{j}_a+\pa{h}_{a+1}\pa{i}_a+\pa{h}_{a+1}\pa{j}_a}\sum_{p=1}^{\mu_a}\sum_{t=0}^{r-1} F_{b;h,p}^{(r+s-1-t)}D_{a;p,j}^{(t)}\\
        +\delta_{a,b+1}(-1)^{\pa{h}_a\pa{k}_b+\pa{h}_a\pa{j}_a+\pa{j}_a\pa{k}_b} \sum_{t=0}^{r-1} F_{b;i,k}^{(r+s-1-t)}D_{a;h,j}^{(t)},
        \end{multline}        
\begin{multline}\label{p706}
 [E_{a;i,j}^{(r)} , F_{b;h,k}^{(s)}]
          =\delta_{a,b}(-1)^{\pa{h}_{a+1}\pa{k}_a+\pa{j}_{a+1}\pa{k}_a+\pa{h}_{a+1}\pa{j}_{a+1}+1}
          \sum_{t=0}^{r+s-1} D_{a;i,k}^{\prime (r+s-1-t)} D_{a+1;h,j}^{(t)},       
          \end{multline}            
\begin{multline}\label{p707}
 [E_{a;i,j}^{(r)} , E_{a;h,k}^{(s)}]
          =(-1)^{\pa{h}_{a}\pa{j}_{a+1}+\pa{j}_{a+1}\pa{k}_{a+1}+\pa{h}_{a}\pa{k}_{a+1}}\times\\
          \big( \sum_{t=1}^{s-1} E_{a;i,k}^{(r+s-1-t)} E_{a;h,j}^{(t)} 
          -\sum_{t=1}^{r-1} E_{a;i,k}^{(r+s-1-t)} E_{a;h,j}^{(t)}  \big),       
\end{multline} 
\begin{multline}\label{p708}
 [F_{a;i,j}^{(r)} , F_{a;h,k}^{(s)}]
          =(-1)^{\pa{h}_{a+1}\pa{j}_{a}+\pa{j}_{a}\pa{k}_{a}+\pa{h}_{a+1}\pa{k}_{a}}\times\\
          \big( \sum_{t=1}^{r-1} F_{a;i,k}^{(r+s-1-t)} F_{a;h,j}^{(t)} 
          -\sum_{t=1}^{s-1} F_{a;i,k}^{(r+s-1-t)} F_{a;h,j}^{(t)}  \big),         
 \end{multline}

\begin{equation}
\label{p709}[E_{a;i,j}^{(r+1)}, E_{a+1;h,k}^{(s)}]-[E_{a;i,j}^{(r)}, E_{a+1;h,k}^{(s+1)}]
=(-1)^{\pa{j}_{a+1}\pa{h}_{a+1}}\delta_{h,j}\sum_{q=1}^{\mu_{a+1}}E_{a;i,q}^{(r)}E_{a+1;q,k}^{(s)}\,,
\end{equation}

\begin{multline}
\label{p710}[F_{a;i,j}^{(r+1)}, F_{a+1;h,k}^{(s)}]-[F_{a;i,j}^{(r)}, F_{a+1;h,k}^{(s+1)}]=\\
(-1)^{\pa{i}_{a+1}(\pa{j}_{a}+\pa{h}_{a+2})+\pa{j}_a\pa{h}_{a+2}+1}\delta_{i,k}\sum_{q=1}^{\mu_{a+1}}F_{a+1;h,q}^{(s)}F_{a;q,j}^{(r)}\,,
\end{multline}

\begin{align}
\label{p711}&[E_{a;i,j}^{(r)}, E_{b;h,k}^{(s)}] = 0
\qquad\qquad\text{\;\;if\;\; $|b-a|>1$ \;\;or\;\; \;if\;\;$b=a+1$ and $h \neq j$},\\[3mm]
\label{p712}&[F_{a;i,j}^{(r)}, F_{b;h,k}^{(s)}] = 0
\qquad\qquad\text{\;\;if\;\; $|b-a|>1$ \;\;or\;\; \;if\;\;$b=a+1$ and $i \neq k$},\\[3mm]
\label{p713}&\big[E_{a;i,j}^{(r)},[E_{a;h,k}^{(s)},E_{b;f,g}^{(\ell)}]\big]+
\big[E_{a;i,j}^{(s)},[E_{a;h,k}^{(r)},E_{b;f,g}^{(\ell)}]\big]=0 \quad \text{if}\,\,\, |a-b|\geq 1,\\[3mm]
\label{p714}&\big[F_{a;i,j}^{(r)},[F_{a;h,k}^{(s)},F_{b;f,g}^{(\ell)}]\big]+
\big[F_{a;i,j}^{(s)},[F_{a;h,k}^{(r)},F_{b;f,g}^{(\ell)}]\big]=0 \quad \text{if}\,\,\, |a-b|\geq 1,\\[3mm]
\label{p715}&\big[\,[E_{a;i,f_1}^{(r)},E_{a+1;f_2,j}^{(1)}]\,,\,[E_{a+1;h,g_1}^{(1)},E_{a+2;g_2,k}^{(s)}]\,\big]=0 \;\;\text{when\;\;}  n\geq 4\, \text{and\;\;} \pa{h}_{a+1}+\pa{j}_{a+2}=1,\\[3mm]
\label{p716}&\big[\,[F_{a;i,f_1}^{(r)},F_{a+1;f_2,j}^{(1)}]\,,\,[F_{a+1;h,g_1}^{(1)},F_{a+2;g_2,k}^{(s)}]\,\big]=0 \;\;\text{when\;\;} n\geq 4\, \text{and\;\;} \pa{j}_{a+1}+\pa{h}_{a+2}=1. 
\end{align}}
If $D_{a;i,j}^{(r)}$ appears on the left-hand side of the equation, then it holds for all $1\leq i,j\leq \mu_{a}$ and all $r\geq 0$; if $E_{a;h,k}^{(s)}$ appears on the left-hand side of the equation, then it holds for all $1\leq h\leq \mu_a, 1\leq k\leq \mu_{a+1}$ and all $s\geq 1$; if $F_{a;f,g}^{(\ell)}$ appears on the left-hand side of the equation, then it holds for all $1\leq g\leq \mu_a, 1\leq f\leq \mu_{a+1}$ and all $\ell\geq 1$.
\end{proposition}
\begin{proof}
The first three relations follow from Proposition \ref{dd0}. By Proposition~\ref{n=2} and Lemma \ref{3be}--Lemma \ref{extra}, one can show that the identities hold in smaller Yangians, for example, $Y_{(\mu_2,\mu_3)}$. Then we apply the injective homomorphisms $\psi_{p_1|q_1}=\psi_{\mu_1}$ to these identities so that the corresponding identities also hold in bigger Yangians, for example, $Y_{(\mu_1,\mu_2,\mu_3)}$. Repeating this process and we may eventually deduce that all these relations, in series forms, hold in $Y_{\mu}$.

Finally, the following identity converts the relations from series form into the desired form:
\[
\dfrac{S(v)-S(u)}{u-v}=\sum_{r,s\ge 1}S^{(r+s-1)}u^{-r}v^{-s},
 \]
for any formal series $S(u)=\sum_{r\ge 0}S^{(r)}u^{-r}$.
\end{proof}

Our main theorem is that the above relations are enough for a set of defining relations of $Y_\mu(\so)$.
\begin{theorem}\label{Pg} 
Let $\mu=(\mu_1,\ldots,\mu_n)$ be a composition of $M+N$ and $\so$ be a $0^M1^N$-sequence. Associated to this $\mu$ and $\so$, the super Yangian $Y_\mu(\so)$ is generated by the parabolic generators
\begin{align*}
&\lbrace D_{a;i,j}^{(r)}, D_{a;i,j}^{\prime(r)} \,|\, 1\leq a\leq n, 1\leq i,j\leq \mu_a, r\geq 0\rbrace,\\
&\lbrace E_{a;i,j}^{(r)} \,|\, 1\leq a< n, 1\leq i\leq \mu_a, 1\leq j\leq\mu_{a+1}, r\geq 1\rbrace,\\
&\lbrace F_{a;i,j}^{(r)} \,|\, 1\leq a< n, 1\leq i\leq\mu_{a+1}, 1\leq j\leq \mu_a, r\geq 1\rbrace,
\end{align*}
subject only to the relations (\ref{p701})$-$(\ref{p716}).
\end{theorem}

The remaining part of this article is devoted to the proof of our main theorem, which is built on several technical propositions and lemmas.

Let $\widehat{Y}_{\mu}$ denote the abstract superalgebra generated by the elements and relations as in the statement of Theorem~\ref{Pg}, where the parities of the generator are given explicitly by (\ref{pad})-(\ref{paf}). We may further define all the other $E_{a,b;i,j}^{(r)}$ and $F_{b,a;i,j}^{(r)}$ in $\widehat{Y}_{\mu}$ by the relations (\ref{ter}), and it is straightforward to check that these definitions are independent of the choices of $k$ as in \cite[p.22]{BK1}. Let $\Gamma$ be the map
\[
\Gamma: \widehat{Y}_{\mu}\longrightarrow Y_{\mu}
\]
sending every element in $\widehat{Y}_{\mu}$ into the element in $Y_{\mu}$ with the same notation. By Theorem~\ref{gendef} and Proposition~\ref{srlns}, the map $\Gamma$ is a surjective superalgebra homomorphism. Therefore, it remains to prove that $\Gamma$ is injective.

The injectivity of $\Gamma$ is proved similar to the arguments in \cite{BK1, Go, Pe1}. We first find a spanning set for $\widehat{Y}_{\mu}$, and then show that the image of this spanning set under $\Gamma$ is linearly independent in $Y_{\mu}$.

Let $\widehat{Y}^0_{\mu}$ (respectively, $\widehat{Y}^+_{\mu}$, $\widehat{Y}^-_{\mu}$) denote the subalgebras of
$\widehat{Y}_{\mu}$ generated by the elements $\lbrace D_{a;i,j}^{(r)}\rbrace$ (respectively, $\lbrace E_{a,b;i,j}^{(r)}\rbrace$, $\lbrace F_{b,a;i,j}^{(r)}\rbrace$).
Define a filtration on $\widehat{Y}_{\mu}$ (on $\widehat{Y}^0_{\mu}$, $\widehat{Y}^+_{\mu}$ and $\widehat{Y}_{\mu}^-$ as well) by setting
\begin{equation*}
\text{deg}(D_{a;i,j}^{(r)})=\text{deg}(E_{a,b;i,j}^{(r)})=\text{deg}(F_{b,a;i,j}^{(r)})=r-1,\qquad\text{for all}\;\; r\ge 1,
\end{equation*}
and denote the associated graded superalgebra by $\gr\widehat{Y}_{\mu}$.
Let $\ovl{E}_{a,b;i,j}^{(r)}$ denote the image of $E_{a,b;i,j}^{(r)}$ in the graded superalgebra $\gr_{r-1}\widehat{Y}_{\mu}^+$.

\begin{lemma} The following identities hold in $\gr\widehat{Y}_{\mu}^+$ for all $r,s,t\geq 1$:
 \begin{enumerate}
  \item[(a)] \begin{equation}\label{L717}
              [\ovl{E}_{a,a+1;i,j}^{(r)},\ovl{E}_{b,b+1;h,k}^{(s)}]=0,\;\text{if}\;|a-b|\ne 1,
             \end{equation}
  \item[(b)] \begin{equation}\label{L718}
              [\ovl{E}_{a,a+1;i,j}^{(r+1)},\ovl{E}_{b,b+1;h,k}^{(s)}]=
              [\ovl{E}_{a,a+1;i,j}^{(r)},\ovl{E}_{b,b+1;h,k}^{(s+1)}],\; \text{if}\; |a-b|=1,
             \end{equation}
  \item[(c)] \begin{equation}\label{L719}
              \big[\ovl{E}_{a,a+1;i,j}^{(r)},[\ovl{E}_{a,a+1;h,k}^{(s)},\ovl{E}_{b,b+1;f,g}^{(t)}]\big]=
              -\big[\ovl{E}_{a,a+1;i,j}^{(s)},[\ovl{E}_{a,a+1;h,k}^{(r)},\ovl{E}_{b,b+1;f,g}^{(t)}]\big],
             \end{equation}
             \; \text{if}\; $|a-b|=1$,
  \item[(d)] \begin{multline}\label{L720}
              \ovl{E}_{a,b;i,j}^{(r)}=(-1)^{\pa{h}_{b-1}}[\ovl{E}_{a,b-1;i,h}^{(r)},\ovl{E}_{b-1,b;h,j}^{(1)}]
              =(-1)^{\pa{k}_{a+1}}[\ovl{E}_{a,a+1;i,k}^{(1)},\ovl{E}_{a+1,b;k,j}^{(r)}],
              \end{multline}
              for all $b>a+1$ and any $1\leq h\leq\mu_{b-1}$, $1\leq k\leq\mu_{a+1}$.
 \end{enumerate}
Here (\ref{L717}) and (\ref{L718}) hold for all $1\leq i\leq \mu_a$, $1\leq j\leq \mu_{a+1}$, $1\leq h\leq \mu_b$, $1\leq k\leq \mu_{b+1}$; (\ref{L719}) holds for all $1\leq i,h\leq \mu_a$, $1\leq j,k\leq \mu_{a+1}$, $1\leq f\leq \mu_b$, $1\leq g\leq \mu_{b+1}$; (\ref{L720}) holds for all $1\leq i\leq \mu_a$, $1\leq j\leq \mu_b$. 
\end{lemma}

\begin{proof}
(\ref{L717}) and (\ref{L718}) follow from (\ref{p711}) and (\ref{p709}), while (\ref{L719}) follows from (\ref{p713}). The first equality of (\ref{L720}) follows from (\ref{ter}), while the second one can be deduced from the first equality by super Jacobi identity, (\ref{L718}) and induction on $b-a$.
\end{proof}

\begin{lemma}
The following identities hold in $\gr\widehat{Y}_{\mu}^+$ for all $r,s\geq 1$:
 \begin{enumerate}
  \item[(a)]
    \begin{equation}\label{LL725}
     [\ovl{E}_{a,a+2;i,j}^{(r)},\ovl{E}_{a+1,a+2;h,k}^{(s)}]=0,\;\; \text{for all}\;\; 1\leq a\leq n-2,
    \end{equation}
  \item[(b)]
    \begin{equation}\label{LL726}
     [\ovl{E}_{a,a+1;i,j}^{(r)},\ovl{E}_{a,a+2;h,k}^{(s)}]=0,\;\; \text{for all}\;\; 1\leq a\leq n-2,
    \end{equation}
  \item[(c)]
    \begin{equation}\label{L723}
     [\ovl{E}_{a,a+2;i,j}^{(r)},\ovl{E}_{a+1,a+3;h,k}^{(s)}]=0,\;\; \text{for all}\;\; 1\leq a\leq n-3,
    \end{equation}
  \item[(d)]
    \begin{equation}\label{L724}
     [\ovl{E}_{a,b;i,j}^{(r)},\ovl{E}_{c,c+1;h,k}^{(s)}]=0,\;\; \text{for all}\;\; 1\leq a<c<b\leq n.
    \end{equation}
 \end{enumerate}
Here (\ref{LL725}) holds for all $1\leq i\leq \mu_a$, $1\leq h\leq \mu_{a+1}$, $1\leq j,k\leq \mu_{a+2}$; (\ref{LL726}) holds for all $1\leq i,h\leq \mu_a$, $1\leq j\leq \mu_{a+1}$, $1\leq j,k\leq \mu_{a+2}$; (\ref{L723}) holds for all $1\leq i\leq \mu_a$, $1\leq h\leq \mu_{a+1}$, $1\leq j\leq \mu_{a+2}$, $1\leq k\leq \mu_{a+3}$; (\ref{L724}) holds for all $1\leq i\leq \mu_a$, $1\leq j\leq \mu_b$, $1\leq h\leq \mu_c$, $1\leq k\leq \mu_{c+1}$. 
\end{lemma}

\begin{proof}
Similar to the proof in \cite[Lemma 8.3]{Pe1} so we only show (c) in detail here since it is the place that we actually use the super Serre relations.

Assume first that $\pa{h}_{a+1}+\pa{j}_{a+2}=0$. 
Applying (\ref{L720}) on the left-hand side of (\ref{L723}) and using the super Jacobi identity, we have
\begin{align*}
 [\ovl{E}_{a,a+2;i,j}^{(r)}&,\ovl{E}_{a+1,a+3;h,k}^{(s)}]=\\
 &(-1)^{\pa{h}_{a+1}+\pa{j}_{a+2}}\big[\,[\ovl{E}_{a,a+1;i,h}^{(r)},\ovl{E}_{a+1,a+2;h,j}^{(1)}]
 \,,\,[\ovl{E}_{a+1,a+2;h,j}^{(1)},\ovl{E}_{a+2,a+3;j,k}^{(s)}]\,\big]\\
&=(-1)^{\pa{h}_{a+1}+\pa{j}_{a+2}}\Big\lbrace\Big[\,\big[\,[\ovl{E}_{a,a+1;i,h}^{(r)},\ovl{E}_{a+1,a+2;h,j}^{(1)}],
 \ovl{E}_{a+1,a+2;h,j}^{(1)}\big],\ovl{E}_{a+2,a+3;j,k}^{(s)}\Big]\\
&+\varepsilon\Big[\ovl{E}_{a+1,a+2;h,j}^{(1)},
 \big[\,[\ovl{E}_{a,a+1;i,h}^{(r)},\ovl{E}_{a+1,a+2;h,j}^{(1)}],\ovl{E}_{a+2,a+3;j,k}^{(s)}\big]\,\Big]\Big\rbrace,
\end{align*}
where $\varepsilon=(-1)^{(\pa{i}_{a}+\pa{h}_{a+1})(\pa{h}_{a+1}+\pa{j}_{a+2})}$.
By (\ref{L719}), the first term is zero. Using the super Jacobi identity, (\ref{L717}) and (\ref{L720}) again, we may deduce that the above equals to
\begin{align*}
&\varepsilon\Big[\ovl{E}_{a+1,a+2;h,j}^{(1)},
 \big[\,[\ovl{E}_{a,a+1;i,h}^{(r)},\ovl{E}_{a+1,a+2;h,j}^{(1)}],\ovl{E}_{a+2,a+3;j,k}^{(s)}\big]\,\Big]\\
=&\varepsilon\Big[\ovl{E}_{a+1,a+2;h,j}^{(1)},
 \big[\ovl{E}_{a,a+1;i,h}^{(r)},[\ovl{E}_{a+1,a+2;h,j}^{(1)},\ovl{E}_{a+2,a+3;j,k}^{(s)}]\,\big]\,\Big]+0\\
=&\varepsilon\big[\,[\ovl{E}_{a+1,a+2;h,j}^{(1)},\ovl{E}_{a,a+1;i,h}^{(r)}],
[\ovl{E}_{a+1,a+2;h,j}^{(1)},\ovl{E}_{a+2,a+3;j,k}^{(s)}]\,\big]+0\\
=&(-1)\varepsilon^2\big[\,[\ovl{E}_{a,a+1;i,h}^{(r)},\ovl{E}_{a+1,a+2;h,j}^{(1)}]
,[\ovl{E}_{a+1,a+2;h,j}^{(1)},\ovl{E}_{a+2,a+3;j,k}^{(s)}]\,\big]\\
=&(-1)^{1+\pa{h}_{a+1}+\pa{j}_{a+2}}[\ovl{E}_{a,a+2;i,j}^{(r)},\ovl{E}_{a+1,a+3;h,k}^{(s)}].
\end{align*}
By our assumption, $\pa{h}_{a+1}+\pa{j}_{a+2}=0$ and we have done. 

Assume on the other hand that $\pa{h}_{a+1}+\pa{j}_{a+2}=1$. Similarly, we apply (\ref{L720}) on the left-hand side of (\ref{L723}) to obtain
\begin{align*}
 [\ovl{E}_{a,a+2;i,j}^{(r)}&,\ovl{E}_{a+1,a+3;h,k}^{(s)}]=\\
 &(-1)^{\pa{h}_{a+1}+\pa{j}_{a+2}}\big[\,[\ovl{E}_{a,a+1;i,h}^{(r)},\ovl{E}_{a+1,a+2;h,j}^{(1)}]
 \,,\,[\ovl{E}_{a+1,a+2;h,j}^{(1)},\ovl{E}_{a+2,a+3;j,k}^{(s)}]\,\big]\\
=&-\big[\,[\ovl{E}_{a,a+1;i,h}^{(r)},\ovl{E}_{a+1,a+2;h,j}^{(1)}]
 \,,\,[\ovl{E}_{a+1,a+2;h,j}^{(1)},\ovl{E}_{a+2,a+3;j,k}^{(s)}]\,\big],
\end{align*}
which is zero directly by (\ref{p715}). 
\end{proof}

The following lemma, generalizing \cite[Lemma 6.7]{BK1} and \cite[(8.1)]{Pe1}, plays a crucial role in the proof of Theorem~\ref{Pg}.
\begin{lemma}\label{injeq}
For all $1\leq a\leq b\leq n$, $1\leq c\leq d\leq n$, $r,s\geq 0$ and
all $1\leq i\leq \mu_a$, $1\leq j\leq \mu_b$, $1\leq h\leq \mu_c$, $1\leq k\leq\mu_d$, we have
\begin{multline*}
[\ovl{E}_{a,b;i,j}^{(r)},\ovl{E}_{c,d;h,k}^{(s)}]=(-1)^{\pa{j}_b\pa{h}_c}\delta_{b,c}\delta_{h,j}\ovl{E}_{a,d;i,k}^{(r+s-1)}\\
-(-1)^{\pa{i}_a\pa{j}_b+\pa{i}_a\pa{h}_c+\pa{j}_b\pa{h}_c}\delta_{a,d}\delta_{i,k}\ovl{E}_{c,b;h,j}^{(r+s-1)}.
\end{multline*}
\end{lemma}
\begin{proof}
 Without loss of generality, we may assume that $a\leq c$. The proof is divided into 7 cases and we discuss them one by one.
\begin{description}
\item[Case 1.] $a<b<c<d$:\\
It follows directly from (\ref{L717}) and (\ref{L720}) that the bracket in Lemma \ref{injeq} is zero.

\item[Case 2.] $a<b=c<d$:\\
By (\ref{L718}) and (\ref{L720}), we have
\begin{equation}\label{L725}
[\ovl{E}^{(r+1)}_{b-1,b;i_1,j}\,,\,\ovl{E}^{(s+1)}_{b,b+1;h,k_1}]
=[\ovl{E}^{(r+s+1)}_{b-1,b;i_1,j}\,,\,\ovl{E}^{(1)}_{b,b+1;h,k_1}]
=\delta_{h,j}(-1)^{\pa{h}_b}\ovl{E}^{(r+s+1)}_{b-1,b+1;i_1,k_1}.
\end{equation}
Note that when $h\neq j$, the bracket is zero by (\ref{61b}) and hence the term $\delta_{h,j}$ shows up.
Taking brackets on both sides of (\ref{L725}) with the elements
\[
\ovl{E}^{(1)}_{b+1,b+2;k_1,k_2}, \ovl{E}^{(1)}_{b+2,b+3;k_2,k_3}, \cdots , \ovl{E}^{(1)}_{d-1,d;k_{d-b+1},k}
\]
from the right then using (\ref{L717}), (\ref{L720}) and the super Jacobi identity, we deduce that
\begin{equation}\label{L726}
[\ovl{E}^{(r+1)}_{b-1,b;i_1,j}\,,\,\ovl{E}^{(s+1)}_{b,d;h,k}]=\delta_{h,j}(-1)^{\pa{h}_b}\ovl{E}^{(r+s+1)}_{b-1,d;i_1,k}\,.
\end{equation}
Taking brackets on both sides of (\ref{L726}) with the elements
\[
\ovl{E}^{(1)}_{b-2,b-1;i_2,i_1}, \ovl{E}^{(1)}_{b-3,b-2;i_3,i_2},\cdots,\ovl{E}^{(1)}_{a,a+1;i,i_{b-a-1}}
\]
from the left and using exactly the same method as above, we have
\[
[\ovl{E}^{(r)}_{a,b;i,j}\,,\,\ovl{E}^{(s)}_{b,d;h,k}]=\delta_{h,j}(-1)^{\pa{h}_b}\ovl{E}^{(r+s-1)}_{a,d;i,k},
\;\text{as desired}.
\]
\item[Case 3.] $a<c<b=d$:\\
Using the super Jacobi identity together with (\ref{L720}) and (\ref{L724}), we have
\begin{align*}
[\ovl{E}^{(r)}_{a,b;i,j}&,\ovl{E}^{(s)}_{c,b;h,k}]
=\big[\ovl{E}^{(r)}_{a,b;i,j},(-1)^{\pa{f_1}_{c+1}}[\ovl{E}^{(1)}_{c,c+1;h,f_1},\ovl{E}^{(s)}_{c+1,b;f_1,k}]\,\big]\\
&=(-1)^{\pa{f_1}_{c+1}}\big[\,[\ovl{E}^{(r)}_{a,b;i,j}\,,\,\ovl{E}^{(1)}_{c,c+1;h,f_1}],\ovl{E}^{(s)}_{c+1,b;f_1,k}\big]\\
&\qquad\qquad\qquad\pm(-1)^{\pa{f_1}_{c+1}}\big[\ovl{E}^{(1)}_{c,c+1;h,f_1}\,,\,[\ovl{E}^{(r)}_{a,b;i,j}\,,\,\ovl{E}^{(s)}_{c+1,b;f_1,k}]\,\big]\\
&=0\pm (-1)^{\pa{f_1}_{c+1}}\big[\ovl{E}^{(1)}_{c,c+1;h,f_1}\,,\,[\ovl{E}^{(r)}_{a,b;i,j}\,,\,\ovl{E}^{(s)}_{c+1,b;f_1,k}]\,\big]\\
&=\cdots =\pm \Big[\ovl{E}^{(1)}_{c,c+1;h,f_1}\,,\,[\ovl{E}^{(1)}_{c+1,c+2;f_1,f_2},\ldots,
[\ovl{E}^{(r)}_{a,b;i,j}\,,\,\ovl{E}^{(s)}_{b-1,b;f_{b-1-c},k}]\,\big]\cdots\Big].
\end{align*}
The bracket $[\ovl{E}^{(r)}_{a,b;i,j}\,,\,\ovl{E}^{(s)}_{b-1,b;f_{b-1-c},k}]$ in the middle is zero by (\ref{L724}).

\item[Case 4.] $a<c<d<b$:\\
Using the same technique as in Case 3, we have
\begin{align*}
[\ovl{E}^{(r)}_{a,b;i,j}&,\ovl{E}^{(s)}_{c,d;h,k}]
=\big[\ovl{E}^{(r)}_{a,b;i,j},(-1)^{\pa{f_1}_{c+1}}[\ovl{E}^{(1)}_{c,c+1;h,f_1},\ovl{E}^{(s)}_{c+1,d;f_1,k}]\,\big]\\
&=(-1)^{\pa{f_1}_{c+1}}\big[\,[\ovl{E}^{(r)}_{a,b;i,j}\,,\,\ovl{E}^{(1)}_{c,c+1;h,f_1}],\ovl{E}^{(s)}_{c+1,d;f_1,k}]\,\big]\\
&\quad\pm(-1)^{\pa{f_1}_{c+1}}\big[\ovl{E}^{(1)}_{c,c+1;h,f_1}\,,\,[\ovl{E}^{(r)}_{a,b;i,j}\,,\,\ovl{E}^{(s)}_{c+1,d;f_1,k}]\,\big]\\
&=0\pm(-1)^{\pa{f_1}_{c+1}}\big[\ovl{E}^{(1)}_{c,c+1;h,f_1}\,,\,[\ovl{E}^{(r)}_{a,b;i,j}\,,\,\ovl{E}^{(s)}_{c+1,d;f_1,k}]\,\big]\\
&=\cdots =\pm \Big[\ovl{E}^{(1)}_{c,c+1;h,f_1}\,,\,\big[\ovl{E}^{(1)}_{c+1,c+2;f_1,f_2}\,,\ldots,[\ovl{E}^{(r)}_{a,b;i,j}\,,\,\ovl{E}^{(s)}_{d-1,d;f_{d-1-c},k}]\,\big]\cdots\Big].
\end{align*}
Following from (\ref{L724}) again, the bracket $[\ovl{E}^{(r)}_{a,b;i,j}\,,\,\ovl{E}^{(s)}_{d-1,d;f_{d-1-c},k}]$ vanishes.

\item[Case 5.] $a<c<b<d$:\\
 We prove this case by induction on $d-b\geq 1$. When $d-b=1$, we have
 \begin{multline*}
  [\ovl{E}^{(r)}_{a,b;i,j},\ovl{E}^{(s)}_{c,b+1;h,k}]
  =\big[\ovl{E}^{(r)}_{a,b;i,j},(-1)^{\pa{j}_{b+1}}[\ovl{E}^{(s)}_{c,b;h,j},\ovl{E}^{(1)}_{b,b+1;j,k}]\,\big]\\
  =(-1)^{\pa{j}_{b+1}}\big[\,[\ovl{E}^{(r)}_{a,b;i,j}\,,\,\ovl{E}^{(s)}_{c,b;h,j}]\,,\,\ovl{E}^{(1)}_{b,b+1;j,k}\big]
  \pm \big[\ovl{E}^{(s)}_{c,b;h,j}\,,\,[\ovl{E}^{(r)}_{a,b;i,j}\,,\,\ovl{E}^{(1)}_{b,b+1;j,k}]\,\big].
 \end{multline*}
Now the bracket in the first term is zero by Case 3, and we may rewrite the whole second term as
$\pm[\ovl{E}^{(r)}_{a,b+1;i,k},\ovl{E}^{(s)}_{c,b;h,j}]$, which is zero by Case 4.

Assume that $d-b>1$, then $d-1>b$. By (\ref{L720}), the bracket equals to
\begin{align*}
[&\ovl{E}^{(r)}_{a,b;i,j},\ovl{E}^{(s)}_{c,d;h,k}]
=\big[\ovl{E}^{(r)}_{a,b;i,j}\,,\,(-1)^{\pa{f}_{d-1}}[\ovl{E}^{(s)}_{c,d-1;h,f}\,,\,\ovl{E}^{(1)}_{d-1,d;f,k}]\,\big]\\
&=(-1)^{\pa{f}_{d-1}}\big[\,[\ovl{E}^{(r)}_{a,b;i,j}\,,\,\ovl{E}^{(s)}_{c,d-1;h,f}]\,,\,\ovl{E}^{(1)}_{d-1,d;f,k}\big]
\pm\big[\ovl{E}^{(s)}_{c,d-1;h,f}\,,\,[\ovl{E}^{(r)}_{a,b;i,j}\,,\,\ovl{E}^{(1)}_{d-1,d;f,k}]\,\big].
\end{align*}
The first term is zero by the induction hypothesis, while the second term is zero as well by Case 1.

\item[Case 6.] $a=c<b<d$:
 \begin{align*}
 [\ovl{E}^{(r)}_{a,b;i,j},\ovl{E}^{(s)}_{a,d;h,k}]
 &=\big[\ovl{E}^{(r)}_{a,b;i,j}\,,\,(-1)^{\pa{f}_{a+1}}[\ovl{E}^{(1)}_{a,a+1;h,f}\,,\,\ovl{E}^{(s)}_{a+1,d;f,k}]\,\big]\\
 &=(-1)^{\pa{f}_{a+1}}\big[\,[\ovl{E}^{(r)}_{a,b;i,j}\,,\,\ovl{E}^{(1)}_{a,a+1;h,f}]\,,\,\ovl{E}^{(s)}_{a+1,d;h,k}\big]\\
 &\quad \pm\big[\ovl{E}^{(1)}_{a,a+1;h,f}\,,\,[\ovl{E}^{(r)}_{a,b;i,j}\,,\,\ovl{E}^{(s)}_{a+1,d;f,k}]\,\big].
 \end{align*}
 Note that $[\ovl{E}^{(r)}_{a,b;i,j}\,,\,\ovl{E}^{(s)}_{a+1,d;f,k}]=0$ by Case 5. Hence it suffices to show that
 \begin{equation}\label{L727}
 [\ovl{E}^{(r)}_{a,b;i,j}\,,\,\ovl{E}^{(1)}_{a,a+1;h,f}]=0, \qquad \text{for all}\quad b>a.
 \end{equation}
 We prove (\ref{L727}) by induction on $b-a\geq 1$. When $b-a=1$, it follows from (\ref{L717}). Now
 assume $b-a>1$. By (\ref{L720}), we have
 \begin{align*}
 [\,\ovl{E}^{(r)}_{a,b;i,j}\,,\,\ovl{E}^{(1)}_{a,a+1;h,f}\,]
 &= \big[\,(-1)^{\pa{g}_{b-1}}\,[\,\ovl{E}^{(r)}_{a,b-1;i,g}\,,\,\ovl{E}^{(1)}_{b-1,b;g,j}\,]\,,\,\ovl{E}^{(1)}_{a,a+1;h,f}\big]\\
 &=(-1)^{\pa{g}_{b-1}}\,\big[\ovl{E}^{(r)}_{a,b-1;i,g}\,,\,[\ovl{E}^{(1)}_{b-1,b;g,j}\,,\,\ovl{E}^{(1)}_{a,a+1;h,f}]\,\big]\\
 &\quad \pm\,\big[\ovl{E}^{(1)}_{b-1,b;g,j}\,,\,[\ovl{E}^{(r)}_{a,b-1;i,g}\,,\,\ovl{E}^{(1)}_{a,a+1;h,f}]\,\big].
 \end{align*}
 Note that $[\ovl{E}^{(r)}_{a,b-1;i,g}\,,\,\ovl{E}^{(1)}_{a,a+1;h,f}]=0$ by the induction hypothesis. Also by (\ref{L717}),
 $[\ovl{E}^{(1)}_{b-1,b;g,j}\,,\,\ovl{E}^{(1)}_{a,a+1;h,f}]=0$ unless $b-1=a+1$, in which case,  
 (\ref{L727}) becomes $[\ovl{E}^{(r)}_{a,a+2;i,j}\,,\,\ovl{E}^{(1)}_{a,a+1;h,f}]$, which is zero by (\ref{LL726}).

\item[Case 7.] $a=c<b=d$:\\
 We claim that
 \begin{equation}\label{L728}
 [\ovl{E}^{(r)}_{a,b;i,j}\,,\,\ovl{E}^{(s)}_{a,b;h,k}]=0.
 \end{equation}
 If $b=a+1$, it follows directly from (\ref{L717}). If $b>a+1$, we may expand one term in the bracket of (\ref{L728}) by (\ref{L720}) to deduce that
 \begin{align*}
  [\,\ovl{E}^{(r)}_{a,b;i,j}\,,\,\ovl{E}^{(s)}_{a,b;h,k}\,]
  &=\big[\,(-1)^{\pa{f}_{b-1}}\,[\ovl{E}^{(r)}_{a,b-1;i,f}\,,\,\ovl{E}^{(1)}_{b-1,b;f,j}]\,,\,\ovl{E}^{(s)}_{a,b;h,k}\big]\\
  &=(-1)^{\pa{f}_{b-1}}\,\big[\ovl{E}^{(r)}_{a,b-1;i,f}\,,\,[\ovl{E}^{(1)}_{b-1,b;f,j}\,,\,\ovl{E}^{(s)}_{a,b;h,k}]\,\big]\\
  &\quad\pm\,\big[\ovl{E}^{(1)}_{b-1,b;f,j}\,,\,[\ovl{E}^{(r)}_{a,b-1;i,f}\,,\,\ovl{E}^{(s)}_{a,b;h,k}]\,\big].
 \end{align*}
Note that $[\ovl{E}^{(1)}_{b-1,b;f,j}\,,\,\ovl{E}^{(s)}_{a,b;h,k}]=0$ by Case 3 and $[\ovl{E}^{(r)}_{a,b-1;i,f}\,,\,\ovl{E}^{(s)}_{a,b;h,k}]=0$ by Case~6, which proves (\ref{L728}).
\end{description}
This completes the proof of Lemma \ref{injeq}.
\end{proof}

\begin{proposition}\label{ind1}
$\widehat{Y}_{\mu}$ is spanned as a vector superspace by supermonomials in the elements $\lbrace D_{a;i,j}^{(r)}, E_{a,b;i,j}^{(r)}, F_{b,a;i,j}^{(r)}\rbrace$ taken in a certain fixed order so that $F$'s appear before $D$'s and $D$'s appear before $E$'s.
\end{proposition}

\begin{proof}
Lemma \ref{injeq} implies that the graded algebra $\gr\widehat{Y}_{\mu}^+$ is spanned by supermonomials in $\lbrace\ovl{E}_{a,b;i,j}^{(r)}\rbrace$ in some fixed order and hence $\widehat{Y}_{\mu}^+$ is spanned by supermonomials in $\lbrace E_{a,b;i,j}^{(r)}\rbrace$ in some fixed order as well.

By applying the automorphism $\zeta_{M|N}$, we see that $\widehat{Y}_{\mu}^-$ is spanned by supermonomials in $\lbrace F_{a,b;i,j}^{(r)}\rbrace$ in a certain fixed order as well.

Moreover, $\gr\widehat{Y}^0_{\mu}$ is supercommutative by Proposition \ref{dd0}, and it follows that $\widehat{Y}^0_{\mu}$ is spanned by supermonomials in $\lbrace D_{a;i,j}^{(r)}\rbrace$ in a certain fixed order.

Finally, by the defining relations in Proposition \ref{srlns} and the argument above, we may interchange the order between those $D$'s, $E$'s and $F$'s in a supermonomial such that all the $F$'s appear before all the $D$'s and all the $D$'s appear before all the $E$'s.

As a result, the multiplication map is surjective:
\[
\gr \widehat{Y}^-_{\mu}\otimes \gr\widehat{Y}^0_{\mu}\otimes \gr\widehat{Y}^+_{\mu}\twoheadrightarrow \gr\widehat{Y}_{\mu}
\]
and our proposition is established.
\end{proof}

\begin{proposition}\label{ind2}
The images of the supermonomials in Proposition~\ref{ind1} under $\Gamma$ are linearly independent.
\end{proposition}

\begin{proof}
By Corollary~\ref{Yloop}, we may identify $\gr Y_{M|N}=\gr Y_{\mu}$ with the loop superalgebra $U(\gl_{M|N}[x])$ via
\[
\gr _{r-1}t_{ij}^{(r)}\longmapsto (-1)^{\pa{i}}e_{ij}x^{r-1}.
\]
We consider the following composition
\[
\gr \widehat{Y}_{\mu}^-\otimes \gr \widehat{Y}_{\mu}^0\otimes \gr \widehat{Y}_{\mu}^+\twoheadrightarrow \gr \widehat{Y}_{\mu} \xrightarrow{\Gamma} \gr Y_{\mu}\cong U(\gl_{M|N}[x]).
\]
Let $n_a:=\mu_1+\mu_2+\ldots+\mu_a$ for short. By Proposition~\ref{quasi}, the image of $\ovl{E}_{a,b;i,j}^{(r)}$ (respectively, $\ovl{D}_{a;i,j}^{(r)}$, $\ovl{F}_{b,a;i,j}^{(r)}$) under the above composition map is $(-1)^{\pa{i}_{a}}e_{n_a+i, n_b+j}x^{r-1}$ (respectively, $(-1)^{\pa{i}_{a}}e_{n_a+i,n_a+j}x^{r-1}$, $(-1)^{\pa{i}_{b}}e_{n_b+i, n_a+j}x^{r-1}$ ). By the PBW theorem for $U(\gl_{M|N}[x])$, the image (under the map $\Gamma$) of the set of all supermonomials in the following set
\begin{align*}
&\qquad\big\lbrace \gr_{r-1}\ovl{D}_{a;i,j}^{(r)} \,|\, 1\leq a\leq n, \; 1\leq i,j\leq\mu_a, \, r\geq 1 \big\rbrace \\
&\cup\big\lbrace \gr_{r-1}\ovl{E}_{a,b;i,j}^{(r)} \,|\, 1\leq a<b\leq n, \; 1\leq i\leq\mu_a, 1\leq j\leq\mu_b, \, r\geq 1 \big\rbrace\\
&\cup\big\lbrace \gr_{r-1}\ovl{F}_{b,a;i,j}^{(r)} \,|\, 1\leq a<b\leq n, \; 1\leq i\leq\mu_b, 1\leq j\leq\mu_a, \, r\geq 1 \big\rbrace
\end{align*}
taken in a certain fixed order must be linearly independent in $\gr Y_{\mu}$ and hence Proposition~\ref{ind2} follows.
\end{proof}

\begin{corollary}
The homomorphism $\Gamma:\widehat{Y}_{\mu}\rightarrow Y_{\mu}$ is injective, and Theorem~\ref{Pg} follows.
\end{corollary}
\begin{proof}
We have known that $\Gamma$ is a surjective homomorphism. Now a spanning set for $\widehat{Y}_{\mu}$ is obtained by Proposition~\ref{ind1}, while the image of this spanning set under $\Gamma$ is linearly independent in $Y_{\mu}$ by Proposition~\ref{ind2}. This shows that $\Gamma$ is injective.
\end{proof}

Let $Y_\mu^0$, $Y_\mu^+$ and $Y_\mu^-$ denote the subalgebras of $Y_{\mu}$ generated by all the $D$'s, $E$'s and $F$'s, respectively. The next result follows from the proof of Proposition~\ref{ind1} and the proof of Proposition~\ref{ind2}.
\begin{corollary} We have the PBW bases for the following superalgebras.
\begin{enumerate}
\item[(1)] The set of supermonomials in $\{ D_{a;i,j}^{(r)}\}_{1\leq a\leq n, 1\leq i,j\leq \mu_a, r\geq 1}$ taken in a certain fixed order forms a basis for $Y_\mu^0$.
\item[(2)] The set of supermonomials in $\{ E_{a,b;i,j}^{(r)}\}_{1\leq a<b\leq n, 1\leq i\leq\mu_a,1\leq j\leq\mu_b, r\geq 1}$ taken in a certain fixed order forms a basis for $Y_\mu^+$.
\item[(3)] The set of supermonomials in $\{ F_{b,a;i,j}^{(r)}\}_{1\leq a<b\leq n, 1\leq i\leq \mu_b,1\leq i\leq\mu_a, r\geq 1}$ taken in a certain fixed order forms a basis for $Y_\mu^-$.
\item[(4)] The set of supermonomials in the union of the elements listed in (1), (2) and (3) taken in a certain fixed order forms a basis for $Y_{\mu}$.
\end{enumerate}
\end{corollary}

\subsection*{Acknowledgements}
The author is grateful to Weiqiang Wang and Shun-Jen Cheng for numerous discussions. This work is partially supported by MOST grant 103-2115-M-008-012-MY2 and NCTS Young Theorist Award 2015.

\end{document}